\documentclass[12pt,BCOR=9.25mm,paper=a4]{scrartcl}
\usepackage{ayv_options}
\newcommand{\prd}{\mathscr{P}_2(\rd)}
\newcommand{\pCBrd}{\mathscr{P}_2(\rd\times\rds)}
\newcommand{\pTCBrd}{\mathscr{P}_2(\rd\times\rd\times\rds)}
\newcommand{\pW}[1]{\mathscr{P}_2(#1)}
\newcommand{\LTwo}[1]{\mathscr{L}_{2}(#1)}
\newcommand{\LTwoS}[1]{\mathscr{L}_{2}^*(#1)}
\newcommand{\letterEkapeps}[1]{{#1}_\varkappa^\varepsilon}
\newcommand{\letterkapeps}[3]{{#1}_\varkappa^\varepsilon[#2,#3]}
\newcommand{\pairkapeps}[2]{(\letterkapeps{t}{#1}{#2},\letterkapeps{\nu}{#1}{#2})}
\newcommand{\pairEkapeps}{(\letterEkapeps{t},\letterEkapeps{\nu})}
\title{Upper and lower bounds of the value function for optimal control in the Wasserstein space}
\author{Yurii Averboukh}\address{Krasovskii Institute of Mathematics and Mechanics, \\ & Yekaterinburg, Russia}
\email{ayv@imm.uran.ru} \author{Aleksei Volkov}\address{Krasovskii Institute of Mathematics and Mechanics, \\ & Yekaterinburg, Russia}
\email{volkov@imm.uran.ru} 
\date{}

\begin{document}
	\maketitle

	\begin{abstract}
	This paper explores the application of nonsmooth analysis in the Wasserstein space to finite-horizon optimal control problems for nonlocal continuity equations. We characterize the value function as a strict viscosity solution of the corresponding Bellman equation using the notions of $\varepsilon$-subdifferentials and $\varepsilon$-superdifferentials. The main paper's result is the fact that continuous subsolutions and supersolutions of this Bellman equation yield lower and upper bounds for the value function. These estimates rely on proximal calculus in the space of probability measures and the Moreau–Yosida regularization. Furthermore, the upper estimates provide a family of approximately optimal feedback strategies that realize the concept of proximal aiming.
	\msccode{49J52, 49L25, 46G05, 49K20, 82C22, 93C20, 49N35}
\end{abstract}
	
		\section{Introduction}
	\paragraph{Overview of the main results.} The aim of this paper is twofold. First, we develop nonsmooth analysis on the product of a finite-dimensional space and the space of probability measures (Wasserstein space). Then, we apply this framework to analyze an optimal control problem for a nonlocal continuity equation:
	\begin{equation}\label{intrdct:payoff:main}
		\text{minimize } \int_{s}^T L(t,m_t,u(t)),dt + G(m_T)
	\end{equation}
	subject to
	\begin{equation}\label{intrdct:eq:dynamics}
		\partial_t m_t + \operatorname{div}(f(t,x,m_t,u(t))m_t), \quad u(t) \in U, \quad m_s = \mu.
	\end{equation}
	Our interest in this problem is motivated by applications such as opinion dynamics, drone swarm control, modeling of pedestrian flows, etc. \cite{Piccoli2018, Fetecau2011, Carrillo2010, Naldi2010, Colombo2012, Dogbe2012, Mogilner1999, HUGHES2002507}.
	
	Our contributions to nonsmooth analysis on the product of a finite-dimensional space and the Wasserstein space concern proximal calculus and an analog of the Moreau–Yosida regularization. Since the Wasserstein space is not 
	$\sigma$-compact, it is more convenient to work with 
	$\varepsilon$-sub-/superdifferentials. Accordingly, we introduce directional and proximal 
	$\varepsilon$-sub-/superdifferentials. In our approach, directional sub-/superdifferentials are defined using shifts in the Wasserstein space given by square integrable functions, with the corresponding sub-/supergradients also being square integrable mappings. This concept refines the one proposed in~\cite{Badreddine2021}. In contrast, for proximal sub-/superdifferentials, we adopt the framework of~\cite[\S 10.3]{ambrosio}. Thus, elements of the proximal $\varepsilon$-sub-/superdifferentials consist of a scalar and a distribution over the cotangent bundle. Note that every proximal $\varepsilon$-subgradient induces a directional $\varepsilon$-subgradient, and the same holds for supergradients.
	
	In studying the Moreau–Yosida regularization, we employ the Ekeland principle to establish the existence of a pseudo-minimizer. Analogous to the finite dimensional case, we prove that this pseudo-minimizer yields an element of the proximal 
$\varepsilon$-subdifferential.
	
	To analyze the value function for \eqref{intrdct:payoff:main}, \eqref{intrdct:eq:dynamics}, we first examine the corresponding Bellman equation, which is a Hamilton-Jacobi PDE in the space of probability measures. Our approach relies on directional 
	$\varepsilon$-sub-/superdifferentials and leads to the notion of strict viscosity sub- and supersolutions. This concept originates from~\cite{Crandall1985}, where Hamilton-Jacobi equations were studied in Banach spaces. We revisit the characterization of the value function for the optimal control problem \eqref{intrdct:payoff:main}, \eqref{intrdct:eq:dynamics} slightly extending the results of~\cite{Badreddine2021} by showing that the value function is a strict viscosity solution of the associated Bellman equation (see Theorem~\ref{Bellman:th:characterization}).
	
	The main results concerning the value function involve upper and lower bounds. Here, we relax the definitions of sub- and supersolutions by requiring the corresponding inequalities to hold only for proximal $\varepsilon$-sub-/superdifferentials rather than all directional $\varepsilon$-sub-/supergradients. This leads to formally weaker notions of proximal sub- and supersolutions. We prove that every proximal subsolution of the Bellman equation provides a lower bound for the value function. Moreover, using such a proximal subsolution, one can construct a family of discontinuous feedback strategies ensuring a payoff no worse than this bound (see Theorem~\ref{feedback:th:u_kapeps}). The construction of this family relies heavily on the Moreau–Yosida regularization and its properties. Finally, we show that every proximal supersolution yields an upper bound for the value function. The proof, similar to that for the lower bound, essentially uses of the Moreau–Yosida regularization.

	\paragraph{Relative works.}
The literature on nonsmooth analysis is extensive. We refer to books~\cite{Mordukhovich2005,Penot2013} for the main concepts of this area in Banach spaces.

Analysis on the Wasserstein space is more complex due to its nonlinearity, though directions can be determined either by functions or by plans. Additionally, a probability measure with a finite $p$-th moment can be viewed as a push-forward measure induced by some random variable. This leads to the identification of the Wasserstein space with a quotient space of an appropriate $\mathscr{L}_p$-space. The notion of differentiability can be introduced using the so-called $L$-derivative, which is the gradient of a variational (flat) derivative of a function depending on a probability measure. An equivalent approach involves lifting of a function defined on the Wasserstein space to an appropriate $\mathscr{L}_p$-space~\cite{Cardaliaguet2019}. Several extensions of nonsmooth analysis concepts to the Wasserstein space have been developed. We refer to~\cite{ambrosio,Lanzetti2022,Lanzetti2024,Gangbo2019,Jimenez2020,Jimenez2023,Bonnet2019a}, where analogs of Fr\'{e}chet subdifferentials were introduced and discussed. Fr\'{e}chet subdifferentials can also be defined by lifting the function from the Wasserstein space to a $\mathscr{L}_p$-space (see~\cite{Gangbo2019,Jimenez2023} and references therein for this approach and its comparison with the intrinsic definition of Fr\'{e}chet subdifferentials).

Directional sub- and superdifferentials for the Wasserstein space were first introduced in~\cite{Badreddine2021}. Our approach differs slightly from the one presented there, as we consider test points given only by shifts along functions. The Moreau–Yosida regularization was introduced for the Wasserstein space in~\cite{Kim2013}, where it was applied to a Hamiltonian. Proximal calculus was extended to the Wasserstein space in~\cite{Volkov_Arxiv}, which also clarified the link between the Moreau–Yosida regularization and proximal subgradients under the assumption that the examined function is nonnegative. The results of our paper include proximal calculus on the product of a finite-dimensional space and the Wasserstein space, and we allow for general uniformly continuous functions.

Now, let us turn to the controlled nonlocal continuity equation. Recall that this equation serves as a mathematical model for a system consisting of many identical particles that interact nonlocally and are influenced by an external field. First, we mention papers dealing with necessary optimality conditions in the form of the Pontryagin maximum principle~\cite{Bonnet2021,Bonnet2019a,Bonnet2019,Pogodaev2016,Bongini2017}. The characterization of the value function as a viscosity solution of the corresponding Bellman equation is given in~\cite{Badreddine2021}, though the result is only proved for compactly supported measures. Sensitivity analysis of the value function for such systems was examined in~\cite{Bonnet2022}, and the existence of optimal control for the nonlocal continuity equation was addressed in~\cite{Bonnet2020}.

Note that control theory for the nonlocal continuity equation is closely related to mean field type control theory~\cite{Jimenez2020,Jimenez2023,Averboukh2025}. In this case, the system consists of rational agents who independently choose their controls to optimize a common payoff. Equivalently, one can assume that the entire system is governed by a force that depends discontinuously on the space variable~\cite{Cavagnari2022}. Mean-field-type control theory for deterministic systems is, in turn, closely connected to control problems for McKean-Vlasov equations. The value function for mean field type control is also characterized via a Bellman equation, which is a Hamilton-Jacobi PDE in the space of measures. Additionally, we note that there is extensive literature on viscosity solutions to Hamilton-Jacobi PDEs in measure spaces (see~\cite{Gangbo2019,Bayraktar03042025,Talbi2024} and references therein for the current state of research in this area).

We conclude this brief survey with works~\cite{Clarke1998,Clarke1998a,Clarke1999,Clarke2000}, where the technique of proximal aiming was developed. Our results concerning upper and lower bounds for the value function, as well as the construction of feedback strategies, follow this approach.
	
	\paragraph{Structure of the paper.} The general notation is introduced in Section~\ref{sect:preliminaries}. In Section~\ref{sect:nonsmooth}, we discuss various constructions of nonsmooth analysis for the product of a finite-dimensional space and the Wasserstein space. Here, we also introduce the Moreau–Yosida regularization and derive its properties.
	
	The statement of the optimal control problem for the nonlocal continuity equation is presented in Section~\ref{sect:control}. This section also includes the definition of the value function and its properties. The proof that the value function is the unique viscosity solution of the corresponding Bellman equation is given in Section~\ref{sect:Bellman}.
	
	The following section contains the construction of approximately optimal feedback strategies based on proximal aiming for the nonlocal continuity equation. This construction, in particular, provides an upper bound for the value function. Finally, lower bounds for the value function are derived in Section~\ref{sect:lb}.
	
	\section{General notation}\label{sect:preliminaries}
	\begin{itemize}
		\item If $X_1,\ldots,X_n$ are sets, $i_1,\ldots,i_k\in \{1,\ldots,n\}$, then $\operatorname{p}^{i_1,\ldots,i_k}$ is a projection operator from $X_1\times\ldots\times X_n$ to $X_{i_1}\times\ldots \times X_{i_k}$, i.e., $\operatorname{p}^{i_1,\ldots,i_k}(x_1,\ldots,x_n)\triangleq (x_{i_1},\ldots,x_{i_k})$. Furthermore, $\operatorname{Id}$ stands for the identical operator on $X$.
		\item If $(\Omega,\mathcal{F})$ and $(\Omega',\mathcal{F}')$ are measurable spaces, $h:\Omega\rightarrow\Omega'$ is a $\mathcal{F}/\mathcal{F}'$-measurable mapping, while $m$ is a measure on $\mathcal{F}$, then $h\sharp m$ stands for the push-forward measure of $m$ through $h$ that is measure on $\mathcal{F}'$ such that, for each $\Upsilon\in\mathcal{F}'$,
		\[(h\sharp m)(\Upsilon)\triangleq m(h^{-1}(\Upsilon)).\]
		\item If $(X,\rho_X)$ is a Polish space, $x_*\in X$, $C>0$, then $\mathbb{B}_C(x_*)$ denotes the closed ball of the radius $C$ centered at $x_*$. 
		\item Given $D\subset X$, $x\in X$, we let $\operatorname{dist}(x,D)\triangleq \inf\{\rho_X(x,y):\ y\in D\}$. 
		\item If $(X,\rho_X)$ is a metric space, then $C_b(X)$ is a space of real-valued bounded continuous functions defined on $X$. Moreover, $\operatorname{UC}(X)$ denotes the set of real-valued functions defined on $X$ those are uniformly continuous on each bounded subset of $X$. Notice that, if $\phi\in\operatorname{UC}(X)$, then $\phi$ is bounded on each bounded set. 
		\item $\operatorname{USCB}(X)$ (respectively, $\operatorname{LSCB}(X)$) denotes the space of functions $\phi:X\rightarrow \mathbb{R}$ those are upper semicontinuous (respectively, lower semicontinuous) and bounded from above (respectively, from below) on each bounded set.
		\item The Polish space $(X,\rho)$ is always equipped with the Borel $\sigma$-algebra, denoted by 
		$\mathcal{B}(X)$. The space of all (nonnegative) finite Borel measures on $X$ is denoted by 
		$\mathcal{M}(X)$, while $\mathcal{P}(X)$ represents the set of all Borel probability measures on 
		$X$. Additionally, $\delta_z$ denotes the Dirac measure concentrated at $z$. 
		\item On $\mathcal{M}(X)$, we consider the topology of narrow convergence. In this case, a sequence $\{m_n\}_{n=1}^\infty\subset \mathcal{P}(X)$ converges to a measure $m$ provided that, for each $\phi\in C_b(X)$,
		\[\int_X\phi(x)m_n(dx)\rightarrow\int_X\phi(x)m(dx).\]
		\item (disintegration theorem) If $(X_1,\rho_1)$, $(X_2,\rho_2)$ are Polish spaces, $\pi\in \mathcal{M}(X_1\times X_2)$, $i\in \{1,2\}$, $m\triangleq \operatorname{p}^i\sharp \pi$, then there exists a weakly measurable family of probabilities $(\pi(\cdot|x_i))_{x_i\in X_i}$ such that, for each $\phi\in C_b(X_1\times X_2)$,
		\begin{equation*}\label{prel:inro:disintegration}\int_{X_1\times X_2}\phi(x_1,x_2)\pi(d(x_1,x_2))=\int_{X_i}\int_{X_{3-i}}\phi(x_1,x_2)\pi(dx_{3-i}|x_i)m(dx_i).\end{equation*} Recall that a family of measures $(\pi(\cdot|x_i))_{x_i\in X_i}\subset \mathcal{P}(X_{3-i})$ is weakly measurable provided that, for each function $\phi\in C_b(X_{3-i})$, the mapping 
		\[X_i\ni x_i\mapsto\int_{X_{3-i}}\phi(x_{3-i})\pi(dx_{3-i}|x_i)\] is measurable. Notice that the disintegration is unique $m$-a.e., i.e., if $(\pi'(\cdot|x_i))_{x_i\in X_i}\subset \mathcal{P}(X_{3-i})$ is measurable and satisfies 
		\[\int_{X_1\times X_2}\phi(x_1,x_2)\pi(d(x_1,x_2))=\int_{X_i}\int_{X_{3-i}}\phi(x_1,x_2)\pi'(dx_{3-i}|x_i)m(dx_i),\] then $\pi(\cdot|x_i)=\pi'(\cdot|x_i)$ for $m$-a.e. $x_i\in X_i$.
		\item If $(X,\rho_X)$ and $(Y,\rho_Y)$ are Polish spaces, while $m\in\mathcal{M}(X)$, we denote by $\Lambda(X,m;Y)$ the set of all measures on $X\times Y$ whose marginal distributions in $X$ are equal to the measure $m$, i.e., $\xi\in\Lambda(X,m;Y)$ if and only if
		\[\operatorname{p}^1\sharp \xi=m.\] The space $\Lambda(X,m;Y)$ inherits the narrow topology from $\mathcal{M}(X\times Y)$. Notice that each measure $\xi\in\Lambda(X,m;Y)$ can be represented via its disintegration as a mapping $X\ni x\mapsto\xi(\cdot|x)\in\mathcal{P}(Y)$.

		\item $\mathscr{P}_2(X)$ is the space of all Borel probabilities on $X$ with a finite second moment, i.e.,
		\[\mathscr{P}_2(X)\triangleq \Bigg\{m\in\mathcal{P}(X):\,\int_X\rho_X^2(x,x_*)m(dx)<\infty\Bigg\}.\]
		Here, $x_*$ is a fixed point in $X$, while $X$ is a Banach space, we will assume that $x_*$ is the zero. 
		\item The space $\pW{X}$ is endowed with the second Wasserstein metric defined by the rule: if $m_1,m_2\in\pW{X}$, then 
		\[W_2(m_1,m_2)\triangleq \Bigg[\inf\Bigg\{\int_X\rho_X^2(x_1,x_2)\pi(d(x_1,x_2)):\,\pi\in\Pi(m_1,m_2)\Bigg\}\Bigg]^{1/2}.\] Here, $\Pi(m_1,m_2)$ denotes the set of all plans between $m_1,m_2$, i.e.,
		\[\Pi(m_1,m_2)\triangleq \Big\{\pi\in\mathcal{P}(X_1\times X):\, \operatorname{p}^i\sharp \pi=m_i,\, i=1,2\Big\}.\] 
		Below, with some abuse of notation, we put, for every $\pi\in \Pi(m_1,m_2)$,
		\[\|\pi\|\triangleq \Bigg[\int_{\rd\times\rd}\rho_X^2(x_1,x_2)\pi(d(x_1,x_2))\Bigg]^{1/2}.\]
		\item We assume that $\rd$ consists of $d$-dimensional column vectors; $\rds$ stands for the space of $d$-dimensional row vectors, the upper index $\top$ denotes the transpose operation.
		\item $\overline{\mathbb{R}}$ stands for the extended real line, i.e., $\overline{\mathbb{R}}\triangleq\mathbb{R}\cup\{+\infty,-\infty\}$.
		\item If $\phi:\rd\rightarrow \mathbb{R}$ is differentiable at $x$, then we denote by $\nabla\phi(x)$ the corresponding derivative that is assumed to be a row vector.
		\item If $X$ is subset of a finite dimensional space, then $C^1(X)$ denotes the set of all continuously differentiable functions defined on $X$, while $C^1_c(X)$ denote the set of continuously differentiable functions from $X$ to $\mathbb{R}$ with compact support.
		\item If $m$ is a probability measure on $\rd$, $\LTwo{m}$ (respectively, $\LTwoS{m}$) denotes the set of square integrable functions from $\rd$ to itself (respectively, from $\rd$ to $\rds$). If $p\in \LTwoS{m}$, $v\in \LTwo{m}$, then we put
		\[\langle p,v\rangle_m\triangleq \int_{\rd} s(x)v(x)m(dx).\] Furthermore, if $v\in \LTwo{m}$, then $v^\top\in\LTwoS{m}$ denotes the pointwise transpose of the function $v$. The same notation applies when $p\in \LTwoS{m}$. Notice that, if $v\in \LTwo{m}$,
		\[\|v\|_{\LTwo{m}}=\sqrt{\langle v^\top,v\rangle_m}.\] Analogously, for $p\in \LTwoS{m}$,
		\[\|p\|_{\LTwoS{m}}=\sqrt{\langle p,p^\top\rangle_m}.\]
	\end{itemize}

	\section{Nonsmooth analysis}\label{sect:nonsmooth}
	In what follows, we will work with the product spaces $\mathbb{R}\times\LTwo{\mu}$, $\mathbb{R}\times\LTwoS{\mu}$ for some $\mu\in\prd$. If $(\theta,v)\in \mathbb{R}\times\LTwo{\mu}$, then \[\|(\theta,v)\|_{\mathbb{R}\times\LTwo{\mu}}\triangleq \sqrt{\theta^2+\|v\|_{\LTwo{\mu}}^2}.\] Analogously, for $(a,p)\in \mathbb{R}\times\LTwoS{\mu}$,
	\[\|(a,p)\|_{\mathbb{R}\times\LTwoS{\mu}}\triangleq \sqrt{a^2+\|p\|_{\LTwoS{\mu}}^2}.\]
	\subsection{Sub- and superdifferentials} 
	
	Let $s\in (0,T)$, $\mu\in\prd$, $r>0$, and let $\phi:(s-r,s+r)\times \mathbb{B}_r(\mu)\rightarrow \overline{\mathbb{R}}$ be finite at $\mu$. First, we introduce analogs of the directional (Dini-Hadamard) sub- and superdifferentials (see \cite[\S 4.1.1]{Penot2013} for the definitions in the Banach space case). Due to the fact that $\prd$ is not $\sigma$-compact, we will use $\varepsilon$-sub- and superdifferentials.
	
	\begin{definition}\label{nonsmooth:def:directional:diffe}
	Let $\varepsilon\geq0$.	The directional $\varepsilon$-subdifferential $\partial_{D,\varepsilon}^-\phi(s,\mu)$ of the function $\phi$ at $(s,\mu)$ consists of all pairs $(a,p)\in \mathbb{R}\times \LTwoS{\mu}$ such that, for every $(\theta,v)\in\mathbb{R}\times \LTwo{\mu}$, 
		\begin{equation*}\label{nonsmooth:intro:directional_subdiff}
		\begin{split} \liminf_{h\downarrow 0, (\theta',v')\rightarrow (\theta,v)}\frac{1}{h}\Big[\phi(s+h\theta',(\operatorname{Id}+hv')\sharp \mu)-\phi(s,\mu)&-h(a\theta'+\langle p,v'\rangle_\mu)\Big]\\&\geq -\varepsilon\cdot \|(\theta,v)\|_{\mathbb{R}\times\LTwo{\mu}}.\end{split}
		\end{equation*} 
		
		Analogously, the directional $\varepsilon$-superdifferential $\partial_{D,\varepsilon}^+\phi(s,\mu)$ of the function $\phi$ at $(s,\mu)$ is the set of all pairs $((a,p)\in \mathbb{R}\times \LTwoS{m}) $ satisfying the following condition:
		\begin{equation*} \begin{split} \limsup_{h\downarrow 0, (\theta',v')\rightarrow (\theta,v)}\frac{1}{h}\Big[\phi(s+h\theta',(\operatorname{Id}+hv')\sharp \mu)-\phi(s,\mu)-&h(a\theta'+\langle p,v'\rangle_\mu)\Big]\\&\leq \varepsilon\cdot\|(\theta,v)\|_{\mathbb{R}\times\LTwo{\mu}}.\end{split}
		\end{equation*} 	
		
	\end{definition}
	
	Notice that in the definitions of the directional sub- and superdifferentials, we used directions in $\prd$ determined by squared integrable functions. This leads to a wide class of sub- and superdifferentials. They will be below utilized to characterize the value function in the same spirit as in~\cite{Badreddine2021}. However, when we discuss construction of approximately optimal feedbacks and the upper and lower bounds of the value function it is convenient to narrow the sub- and superdifferentials. Thus, we will consider proximal $\varepsilon$-sub- and superdifferentials assuming additionally that the directions are now determined by plans. This approach follows \cite[Deﬁnition 10.3.1]{ambrosio} where the Frechet subdifferential was considered. Additionally, we refer to \cite{Volkov_Arxiv}, where the proximal $\varepsilon$-subdifferentials were considered for functions defined on $\prd$.
	
	\begin{definition}\label{nonsmooth:def:prox_diff} We say that a  pair $(a,\gamma)\in \mathbb{R}\times\pCBrd$ is a proximal $\varepsilon$-subgradient of $\phi$ at $(s,\mu)$ if the following condition holds: there exists $\alpha>0$ and $\sigma\geq 0$ such that, for every $(t,\nu)\in \mathbb{B}_\alpha(s,\mu)$ and $\beta\in\pTCBrd$ satisfying
	\begin{itemize}
		\item $\operatorname{p}^{1,3}\sharp\beta=\gamma$,
		\item $\operatorname{p}^2\sharp\beta=\nu$,
	\end{itemize} one has that 
	\begin{equation}\label{nonsmooth:ineq_intro:proximal}
		\begin{split}
	\phi(t,\nu)-\phi(s,\mu)\geq a(t-s)+ \int_{\rd\times\rd\times\rds}q(x'-&x)\beta(d(x,x',q))\\-\sigma\cdot\big((t-s)^2+\|\operatorname{p}^{1,2}\sharp\beta\|^2\big)&-\varepsilon\cdot \big(|t-s|^2+W_2^2(\nu,\mu)\big)^{1/2}.\end{split}\end{equation}
		
	The set of all proximal $\varepsilon$-subgradient of the function $\phi$ at $(s,\mu)$ is called a $\varepsilon$-subdifferential of the function $\phi$ at $(s,\mu)$ and is denoted by	$\partial_{P,\varepsilon}^-\phi(s,\mu)$. 	\end{definition}
	
	\begin{definition}\label{nonsmooth:def:prox_superdiff} The proximal $\varepsilon$-superdifferential of the function $\phi$ at $(s,\mu)$ denoted by $\partial_{P,\varepsilon}^+\phi(s,\mu)$ 
	 consists of all pairs $(a,\gamma)\in \mathbb{R}\times\pCBrd$ satisfying the following property: there exists $\alpha\in (0,r]$ and $\sigma\geq 0$ such that, for every $(t,\nu)\in \mathbb{B}_\alpha(s,\nu)$ and $\beta\in\pTCBrd$ with
	\begin{itemize}
		\item $\operatorname{p}^{1,3}\sharp\beta=\gamma$,
		\item $\operatorname{p}^2\sharp\beta=\nu$,
	\end{itemize} one has that 
	\[\begin{split}
	\phi(t,\nu)-\phi(s,\mu)\leq a(t-s)+ \int_{\rd\times\rd\times\rds}q(x'-x)&\beta(d(x,x',q))\\+\sigma\cdot\big((t-s)^2+\|\operatorname{p}^{1,2}\sharp\beta\|^2\big)&+\varepsilon\cdot \big(|t-s|^2+W_2^2(\nu,\mu)\big)^{1/2}.\end{split}\]
Elements of $\partial_{P,\varepsilon}^+\phi(s,\mu)$ are called proximal $\varepsilon$-supergradients of the function $\phi$ at $(s,\mu)$.\end{definition} Notice that $(a,\gamma)\in \partial_{P,\varepsilon}^+\phi(s,\mu)$ if and only if 
\[(-a,(\operatorname{p}^1,-\operatorname{p}^2)\sharp\gamma)\in \partial_{P,\varepsilon}^-(-\phi)(s,\mu).\]

To provide the link between the directional and proximal $\varepsilon$-sub-/superdifferential, we need the following operation. Given $\gamma\in\pCBrd$, the function $\mathscr{b}[\gamma]$ defined by the rule
\[\mathscr{b}[\gamma](x)\triangleq \int_{\rds}q\gamma(dq|x)\] is called a barycenter of $\gamma$. Here $(\gamma(\cdot|x))_{x\in\rd}$ is the disintegration of the probability $\gamma$ w.r.t.\ $ \operatorname{p}^1\sharp\gamma$. Notice that $\mathscr{b}[\gamma]\in\LTwoS{\operatorname{p}^1\sharp\gamma}$.

\begin{proposition}\label{nonsmooth:prop:diff_prox_directional} Let $\gamma\in \partial_{P,\varepsilon}^-\phi(s,\mu)$. Then, $(a,\mathscr{b}[\gamma])\in\partial_{D,\varepsilon}^-\phi(s,\mu)$. Similarly, if $(a,\mathscr{b}[\gamma])\in\partial_{P,\varepsilon}^+\phi(s,\mu)$, then $(a,\mathscr{b}[\gamma])\in\partial_{D,\varepsilon}^+\phi(s,\mu)$.
\end{proposition}
\begin{proof}
	We will prove only the first statement. Let 
	\begin{itemize}
	\item $\alpha$ be the constant from Definition~\ref{nonsmooth:def:prox_diff} corresponding to $\gamma$; 
	\item $(\theta,v), (\theta',v')\in\mathbb{R}\times\LTwo{\mu}$;
	\item $h$ be a sufficiently small number such that 
	$h\|(\theta',v')\|_{\mathbb{R}\times\LTwo{\mu}}\leq \alpha$.
	\end{itemize}
	We put \[\beta\triangleq (\operatorname{p}^1,\operatorname{p}^1+hv'\circ\operatorname{p}^1, \operatorname{p}^2)\sharp\gamma.\] Using the fact that $\gamma\in\partial_{P,\varepsilon}^-\phi(s,\mu)$ and letting in~\eqref{nonsmooth:ineq_intro:proximal} $\nu\triangleq (\operatorname{Id}+hv')\sharp\mu= \operatorname{p}^2\sharp\beta$, we conclude that
		\[
	\begin{split}
		\phi(s+h\theta,(\operatorname{Id}&+hv')\sharp\mu)-\phi(s,\mu)\\\geq ha\theta'&+\int_{\rd\times\rd\times\rds}q(x'-x)\beta(d(x,x',q))\\&-h^2\sigma^2(\theta')^2-\sigma\int_{\rd\times\rd}|x'-x|^2(\operatorname{p}^{1,2}\sharp\beta)(d(x,x'))\\&-\varepsilon\big[h^2(\theta')^2+ W_2^2((\operatorname{Id}+hv')\sharp\mu,\mu)\big]^{1/2}
		\\	\geq ha\theta'&+h\int_{\rd\times\rds}qv(x)\gamma(d(x,q))\\&-h^2\sigma^2(\theta')^2-h^2\sigma\|v'\|_{\LTwo{\mu}}-h\varepsilon \big[(\theta')^2+\|v'\|_{\LTwo{\mu}}^2 \big]^{1/2}.
	\end{split}
	\] 
Since $\int_{\rd\times\rds}qv(x)\gamma(d(x,q))=\int_{\rd}\mathscr{b}[\gamma](x) v(x)\mu(dx)$, dividing both sides of the last inequality by $h$ and passing to the limit when $h\downarrow 0$, $(\theta',v')\rightarrow (\theta,v)$, we obtain
\[(a,\mathscr{b}[\gamma])\in\partial_{D,\varepsilon}^-\phi(s,\mu).\]
\end{proof}

\subsection{Moreau-Yosida regularization}\label{subsect:nonsmooth:MY}
Let $D$ be a bounded subset of $\prd$. We denote $D^{(1)}\triangleq \{\nu\in\prd:\operatorname{dist}(\nu,D)\leq 1\}$. Additionally, let $\phi\in \operatorname{LSCB}([0,T]\times D^{(1)})$.

For every $s\in [0,T]$, $\mu\in D$, $\varkappa>0$, we define the Moreau-Yosida regularization of the function $\phi$ on $D$ by the rule:
\begin{equation}\label{nonsmooth:intro:Moreau_Yosida}
	\begin{split}
		\phi_\varkappa(s,\mu)\triangleq \inf\Bigg\{\phi(t,\nu)+\frac{1}{2\varkappa^2}|t-s|^2+\frac{1}{2\varkappa^2}W_2^2(\mu,\nu):t\in (0,T),\, \nu\in D^{(1)}\Bigg\}.\end{split}
\end{equation}
Clearly,
\begin{equation}\label{nonsmooth:equality:plan_MY}
	\begin{split}
		\phi_\varkappa(s,\mu)= \inf\Bigg\{\phi(t,\nu)+\frac{1}{2\varkappa^2}|t-s|^2+\frac{1}{2\varkappa^2}&\int_{\rd}|x-y|^2\pi(d(x,y)):\\ t&\in (0,T),\, \nu\in D^{(1)}, \, \pi\in\Pi(\nu,\mu)\Bigg\}.\end{split}
\end{equation}

Notice that since a closed ball in $\prd$ is not compact, we can not guarantee the existence of the minimizer in the right-hand side of~\eqref{nonsmooth:intro:Moreau_Yosida}. However, due to the Ekeland variational principle (see~\cite[Theorem 2.1.1]{Borwein2005}, \cite{Ekeland_1974}), given $\varepsilon>0$, there exists a pair $\pairEkapeps=\pairkapeps{s}{\mu}\in [0,T]\times\prd$ satisfying
\begin{equation}\label{nonsmooth:intro_ineq:s_mu_t_nu_alpha_eps}
\begin{split}
	\phi\pairEkapeps+\frac{1}{2\varkappa^2}&|\letterEkapeps{t}-s|^2+\frac{1}{2\varkappa^2}W_2^2(\mu,\letterEkapeps{\nu})\\+&\varepsilon\big(|s-\letterEkapeps{t}|^2+W_2^2(\mu,\letterEkapeps{\nu})\big)^{1/2}\leq 
	\phi(s,\mu)
\end{split}
\end{equation}
 and, for every $(t,\nu)\in[0,T]\times \prd$,
\begin{equation}\label{nonsmooth:intro_ineq:t_nu_alpha_eps}
\begin{split}
	\phi(t,\nu)+\frac{1}{2\varkappa^2}|t-s|^2+\frac{1}{2\varkappa^2}&W_2^2(\mu,\nu)+\varepsilon\big((t-\letterEkapeps{t})^2+W_2^2(\nu,\letterEkapeps{\nu})\big)^{1/2}\\\geq &\phi\pairEkapeps+\frac{1}{2\varkappa^2}|\letterEkapeps{t}-s|^2+\frac{1}{2\varkappa^2}W_2^2(\mu,\letterEkapeps{\nu}).
\end{split}
\end{equation}

\begin{lemma}\label{nonsmooth:lm:bound_1} Let $\phi$ be bounded on $D^{(1)}$, and let \begin{equation}\label{nonsmooth:intro:c_0}c_0\triangleq \sup\big\{|\phi(t,\nu)|:\, t\in [0,T],\,\nu\in D^{(1)}\big\}.\end{equation} Then,
	\[\big[(s-\letterEkapeps{t})^2+W_2^2(\mu,\letterEkapeps{\nu})\big]^{1/2}\leq \varkappa\sqrt{2c_0}. \]
\end{lemma}
\begin{proof} 
	Denoting $A\triangleq \big[(s-\letterEkapeps{t})^2+W_2^2(\mu,\letterEkapeps{\nu})\big]^{1/2}$, from~\eqref{nonsmooth:intro_ineq:s_mu_t_nu_alpha_eps}, we have that
	\[\frac{1}{2\varkappa^2}A^2+\varepsilon A\leq \phi(s,\mu)-\phi\pairEkapeps.\] Hence,
	\begin{equation*}\label{nonsmooth:ineq:quadratic_phi}A^2+2\varkappa^2\varepsilon A\leq 2\varkappa^2[\phi(s,\mu)-\phi\pairEkapeps].\end{equation*} Thus,
	\[A\leq -\varepsilon\varkappa^2+\varkappa\sqrt{\varepsilon^2\varkappa^2+2|\phi(s,\mu)-\phi\pairEkapeps|}.\] Hence,
	\begin{equation}\label{nonsmooth:ineq:pair_estimate}
		\big[(s-\letterEkapeps{t})^2+W_2^2(\mu,\letterEkapeps{\nu})\big]^{1/2}\leq \varkappa\sqrt{2}|\phi(s,\mu)-\phi\pairEkapeps|^{1/2}.
	\end{equation} The boundness of the function $\phi$ gives the statement of the lemma.
\end{proof}

From now, we assume that $\phi\in \operatorname{UC}([0,T]\times D^{(1)})$. Certainly, the function $\phi$ is still bounded by $c_0$ defined by~\eqref{nonsmooth:intro:c_0}. Moreover, one can introduce the modulus of continuity on $[0,T]\times D^{(1)}$ by the rule:
\begin{equation}\label{nonsmooth:intro:modulus}
\begin{split}
	\omega(\kappa)\triangleq \sup\Big\{|\phi(t',&\nu')-\phi(t'',\nu'')|:\\ &t',t''\in [0,T],\, \nu',\nu''\in D^{(1)},\, |t'-t''|^2+W_2^2(\nu',\nu'')\leq \kappa^2\Big\}.\end{split}
\end{equation} We define
\begin{equation}\label{nonsmooth:intro:N}
	\varrho_{1}(\kappa)\triangleq\kappa\sqrt{2}\big(\omega(\kappa\sqrt{2c_0})\big)^{1/2}\wedge \kappa\sqrt{2c_0}.
\end{equation} Notice that $\kappa^{-1}\varrho_{1}(\kappa)\rightarrow 0$ as $\kappa\rightarrow 0$.

\begin{corollary}\label{nonsmooth:corollary:bound_2} The following estimate holds: 
\[\big[(s-\letterEkapeps{t})^2+W_2^2(\mu,\letterEkapeps{\nu})\big]^{1/2}\leq \varrho_{1}(\varkappa). \] 
\end{corollary}
\begin{proof}
	It suffices to notice that $|\phi(s,\mu)-\phi\pairEkapeps|\leq \omega(\varkappa\sqrt{2c_0})$ and use inequality~\eqref{nonsmooth:ineq:pair_estimate}.
\end{proof}

If $\letterEkapeps{\pi}$ is an optimal plan between $\mu$ and $\letterEkapeps{\nu}$, we put
\begin{equation}\label{nonsmooth:intro:a_diff}
\letterEkapeps{a}=\letterkapeps{a}{s}{\mu}\triangleq \varkappa^{-2}(s-\letterEkapeps{t}),
\end{equation}
\begin{equation}\label{nonsmooth:intro:alpha}\letterEkapeps{\gamma}=\letterkapeps{\gamma}{s}{\mu}\triangleq (\operatorname{p}^2,\varkappa^{-2}(\operatorname{p}^1-\operatorname{p}^2)^\top)\sharp \letterEkapeps{\pi}.\end{equation} To simplify notation, we will omit arguments within this section, where $s$ and $\mu$ are assumed to be fixed.

Following~\cite{Jimenez2023}, we introduce the set 
\begin{equation}\label{nonsmooth:intro:dis_minus}\begin{split}
	\operatorname{dis}^-(\mu)\triangleq \Big\{c(F-\operatorname{Id})^\top:\ \  F:\rd\rightarrow\rd\text{ is an optimal transportation}&{}\\\text{map between }\mu\text{ and }F\sharp\mu,\ \ c\in [0,+&\infty)\Big\}.\end{split}\end{equation} Recall that $F$ is an optimal transportation map between $\mu$ and $F\sharp\mu$, the function $(F-\operatorname{Id})$ is called an optimal displacement \cite{ambrosio}. So, $\operatorname{dis}^-(\mu)$ is a cone generated by transposed optimal displacements. Moreover, from \cite[Lemma 4]{Jimenez2020}, it follows that $\mathscr{b}[\letterEkapeps{\gamma}]\in\operatorname{dis}^-(\mu)$.

\begin{lemma}\label{nonsmooth:lm:alpha_prox}
Let $s\in (0,T)$.	If $\varkappa$ satisfies $\varrho_{1}(\varkappa)<1\wedge s\wedge (T-s)$, then 
	\[(\letterEkapeps{a},\letterEkapeps{\gamma})\in\partial_{P,\varepsilon}^-\pairEkapeps.\]
\end{lemma}
\begin{proof} First, notice that, due to the assumption of $s$, $\varkappa$ and $\varepsilon$, $\letterEkapeps{t}\in (0,T)$. At the same time, Corollary~\ref{nonsmooth:corollary:bound_2} gives that $\letterEkapeps{\nu}\in\operatorname{int}(D^{(1)})$.
	
	Let 
	\begin{itemize}
	\item $h<\letterEkapeps{t}\wedge (T-\letterEkapeps{t})\wedge\operatorname{dist}(\letterEkapeps{\nu},\partial D^{(1)});$
	\item $(t,\nu)\in \mathbb{B}_h\pairEkapeps$; 
	\item $\beta\in\pTCBrd$ be such that $\operatorname{p}^{1,3}\sharp\beta=\letterEkapeps{\gamma}$, $\operatorname{p}^2\sharp\beta=\nu$.\end{itemize} We put
	\begin{equation}\label{nonsmooth:intro:prox_proof_varpi}\varpi\triangleq ((\varkappa^2\operatorname{p}^3)^\top-\operatorname{p}^1,\operatorname{p}^1,\operatorname{p}^2)\sharp\beta.\end{equation}
	Notice that $\varpi\in\pW{(\rd)^3}$ is such that $\operatorname{p}^{1,2}\sharp\varpi=\letterEkapeps{\pi}$, $\operatorname{p}^{2,3}\sharp\varpi=\pi$. Since $\operatorname{p}^{1,3}\sharp\varpi$ is a plan between $\mu$ and $\nu$, from~\eqref{nonsmooth:intro_ineq:t_nu_alpha_eps}, we have that 
	\begin{equation*}
		\begin{split}
			\phi(t,\nu)-\phi&\pairEkapeps\\\geq &\frac{1}{2\varkappa^2}\big(|\letterEkapeps{t}-s|^2-|t-s|^2\big)\\&{}\hspace{20pt} +\frac{1}{2\varkappa^2}\int_{(\rd)^3}\big(|x-z|^2-|x-y|^2\big)\varpi(d(x,z,y))\\&{}\hspace{20pt}-\varepsilon\big((t-\letterEkapeps{t})^2+W_2^2(\nu,\letterEkapeps{\nu})\big)^{1/2}.
		\end{split}
	\end{equation*} Thus,
		\begin{equation*}
		\begin{split}
			\phi(t,\nu)-\phi&\pairEkapeps\\\geq &\frac{1}{2\varkappa^2}\big(2(s-\letterEkapeps{t})(t-\letterEkapeps{t})-|t-\letterEkapeps{t}|^2\big)\\&{}\hspace{20pt}+\frac{1}{2\varkappa^2}\int_{(\rd)^3}(2(x-z)^\top(y-z)-|y-z|^2)\varpi(d(x,z,y))\\&{}\hspace{20pt}-\varepsilon\big((t-\letterEkapeps{t})^2+W_2^2(\nu,\letterEkapeps{\nu})\big)^{1/2}.
		\end{split}
	\end{equation*} Using the definition of $\varpi$ (see~\eqref{nonsmooth:intro:prox_proof_varpi}) and Definition~\ref{nonsmooth:def:prox_diff}, we arrive at the statement of the lemma.
\end{proof}

\begin{lemma}\label{nonsmooth:lm:shift} Let $s'\in [0,T]$, $\mu'\in D$, $\pi\in \Pi(\mu,\mu')$ and let $\varpi\in\pW{(\rd)^3}$ be such that
	\begin{itemize}
		\item $\operatorname{p}^{1,3}\sharp\varpi=\pi$;
		\item $\operatorname{p}^{1,2}\sharp\varpi=\letterEkapeps{\pi}$.
	\end{itemize} Then,
	\[\begin{split}	\phi_\varkappa(s',&\mu')\leq \phi_\varkappa(s,\mu)\\
	&+\frac{1}{\varkappa^2}(s-\letterEkapeps{t})(s'-s)+\frac{1}{\varkappa^2}
	\int_{(\rd)^3}\big[(x-z)^\top(x'-x)\big]\varpi(d(x,z,x'))\\
	&+ \frac{1}{2\varkappa^2}\Bigg[|s'-s|^2+\int_{\rd\times\rd}(x'-x)^2\pi(d(x,x'))\Bigg]+\varepsilon\varrho_{2}(\varkappa),
	\end{split}\] where $\kappa^{-1}\varrho_{2}(\kappa)\rightarrow 0$ as $\kappa\rightarrow 0$.
\end{lemma}
\begin{proof}
	By construction of the function $\phi_\varkappa$, we have that
	\[
	\phi_\varkappa(s',\mu')\leq \phi\pairEkapeps+ \frac{1}{2\varkappa^2}\big[|s'-\letterEkapeps{t}|^2+W_2^2(\mu',\letterEkapeps{\nu})\big].\] At the same time, the choice of the probability $\varpi$ gives the inequality
	\[
	\begin{split}
		\phi_\varkappa(s',\mu')\leq \phi&\pairEkapeps+ \frac{1}{2\varkappa^2}\big[|s'-s|^2+|s-\letterEkapeps{t}|^2+2(s-\letterEkapeps{t})(s'-s)\big]\\+&\frac{1}{2\varkappa^2}
		\int_{(\rd)^3}\big[(x'-z)^2+(x-z)^2+(x-z)^\top(x'-x)\big]\varpi(d(x,z,x')).
	\end{split}
	\] Taking into account the conditions on $\varpi$ and the definition of $\letterEkapeps{\pi}$, we conclude that 
	\begin{equation}\label{nonsmooth:ineq:phi_s_mu_prime}\begin{split}
		\phi_\varkappa(s',&\mu')\\\leq \phi&\pairEkapeps+\frac{1}{2\varkappa^2}\big[|s-\letterEkapeps{t}|^2+W_2^2(\letterEkapeps{\nu},\mu)\big]\\
		&+\frac{1}{\varkappa^2}(s-\letterEkapeps{t})(s'-s)+\frac{1}{\varkappa^2}
		\int_{(\rd)^3}\big[(x-z)^\top(x'-x)\big]\varpi(d(x,z,x'))\\
		&+ \frac{1}{2\varkappa^2}\Bigg[|s'-s|^2+\int_{\rd\times\rd}(x'-x)^2\pi(d(x,x'))\Bigg].
		\end{split}
	\end{equation}	Furthermore, let $\{(t_n,\nu_n)\}_{n=1}^\infty\subset [0,T]\times D^{(1)}$ be such that
	\[\phi(t_n,\nu_n)+\frac{1}{2\varkappa^2}\big[|s-t_n|^2+W_2^2(\mu,\nu_n)\big]\leq\phi_\varkappa(s,\mu)+\frac{1}{n}.\] Therefore,
	\[\begin{split}
	\phi(t_n,\nu_n)+&\frac{1}{2\varkappa^2}\big[|s-t_n|^2+W_2^2(\mu,\nu_n)\big]\\ \leq &\phi\pairEkapeps+\frac{1}{2\varkappa^2}\big[|s-\letterEkapeps{t}|^2+W_2^2(\letterEkapeps{\nu},\mu)\big]+\frac{1}{n}.\end{split}\] Notice that, \[\begin{split}
	\big(|s-t_n|^2+&W_2^2(\mu,\nu_n)\big)^{1/2}\\&\geq \big(|s-\letterEkapeps{t}|^2+W_2^2(\mu,\letterEkapeps{\nu})\big)^{1/2}-\big(|\letterEkapeps{t}-t_n|^2+W_2^2(\letterEkapeps{\nu},\nu_n)\big)^{1/2}.\end{split}\] Thus,
	\[\begin{split}
	|s-t_n|^2+W_2^2(&\mu,\nu_n)\\\geq |s-\letterEkapeps{t}&|^2+W_2^2(\mu,\letterEkapeps{\nu}) +|\letterEkapeps{t}-t_n|^2+W_2^2(\letterEkapeps{\nu},\nu_n)\\ &-2(|\letterEkapeps{t}-t_n|^2+W_2^2(\letterEkapeps{\nu},\nu_n))^{1/2}(|s-\letterEkapeps{t}|^2+W_2^2(\mu,\letterEkapeps{\nu}))^{1/2}.\end{split}\]
	If we denote $A_n\triangleq (|\letterEkapeps{t}-t_n|^2+W_2^2(\letterEkapeps{\nu},\nu_n))^{1/2}$, then
	\[
	(A_n)^2-2A_n(|s-\letterEkapeps{t}|^2+W_2^2(\mu,\letterEkapeps{\nu}))^{1/2}\leq 2\varkappa^2|\phi\pairEkapeps-\phi(t_n,\nu_n)|+2\frac{\varkappa^2}{n}.
	\] Furthermore, we use Corollary~\ref{nonsmooth:corollary:bound_2}. Hence,
	\[ (A_n)^2-2\varrho_{1}(\varkappa)A_n\leq 2\varkappa^2|\phi\pairEkapeps-\phi(t_n,\nu_n)|+2\frac{\varkappa^2}{n}. \] This gives that
	\begin{equation}\label{nonsmooth:ineq:t_nu_n_t_nu_eps}A_n\leq 2\varrho_{1}(\varkappa)+2\varkappa\bigg(|\phi\pairEkapeps-\phi(t_n,\nu_n)|^{1/2}+\frac{1}{n^{1/2}}\bigg).\end{equation}
	Now, we recall that $(t_n,\nu_n)\in [0,T]\times D^{(1)}$ and $\pairEkapeps\in [0,T]\times D^{(1)}$. Thus, we estimate the right-hand side in~\eqref{nonsmooth:ineq:t_nu_n_t_nu_eps} and arrive at the inequality
	\[A_n\leq \varrho_{1,n}'(\varkappa)\]
	with \[\varrho_{1,n}'(\kappa)\triangleq 2\varrho_{1}(\kappa)+2\kappa\sqrt{2c_0}+\frac{2\kappa}{\sqrt{n}}.\]
	Using the fact that $\phi$ is uniformly continuous on $[0,T]\times D^{(1)}$, we have that 
	\begin{equation}\label{nonsmooth:ineq:dist_t_n_nu_n}(|\letterEkapeps{t}-t_n|^2+W_2^2(\letterEkapeps{\nu},\nu_n))^{1/2}=A_n\leq \varrho_{2,n}'(\varkappa),\end{equation} where
	\[\varrho_{2,n}'(\kappa)\triangleq 2\varrho_{1}(\kappa)+2\kappa\Big[\omega\big(\varrho_{1,n}'(\kappa)\big)\Big]^{1/2}+\frac{2\kappa}{n^{1/2}}.\]
	Notice that, if $n\rightarrow \infty$, then $\varrho_{2,n}(\cdot)$ converges to the function $\varrho_{2}(\cdot)$ defined by the rule
	\begin{equation*}
		\varrho_{2}(\kappa)\triangleq 2\varrho_{1}(\kappa)+2\varkappa\Big[\omega\big(2\varrho_{1}(\kappa)+2\kappa\sqrt{c_0}\big)\Big]^{1/2}.
	\end{equation*}
	 Moreover, this convergence is uniform on each interval $[0,\delta^0]$. Additionally,
	 \[\kappa^{-1}\varrho_{2}(\kappa)\rightarrow 0\text{ as }\kappa\rightarrow 0.\]

	Due to~\eqref{nonsmooth:intro_ineq:t_nu_alpha_eps}, we have that 
	\[
	\begin{split}
		\phi\pairEkapeps+\frac{1}{2\varkappa^2}\big[&|s-\letterEkapeps{t}|^2+W_2^2(\letterEkapeps{\nu},\mu)\big] \\ \leq \phi(t_n,\nu_n)+&\frac{1}{2\varkappa^2}\big[|s-t_n|^2+W_2^2(\mu,\nu_n)\big]+\varepsilon \big[|\letterEkapeps{t}-t_n|^2+W_2^2(\letterEkapeps{\nu},\nu_n)\big]^{1/2}.
	\end{split}
	\]
	Estimating the term $\big[|\letterEkapeps{t}-t_n|^2+W_2^2(\letterEkapeps{\nu},\nu_n)\big]$ by~\eqref{nonsmooth:ineq:dist_t_n_nu_n}, we obtain
	\[	\begin{split}
		\phi\pairEkapeps+\frac{1}{2\varkappa^2}\big[|s-\letterEkapeps{t}|^2+W_2^2&(\letterEkapeps{\nu},\mu)\big] \\ \leq \phi(t_n,\nu_n)+&\frac{1}{2\varkappa^2}\big[|s-t_n|^2+W_2^2(\mu,\nu_n)\big]+\varepsilon \varrho_{2,n}(\varkappa).
	\end{split}
	\]
	Furthermore, since $\phi(t_n,\nu_n)+\frac{1}{2\varkappa^2}\big[|s-t_n|^2+W_2^2(\mu,\nu_n)\rightarrow \phi_\varkappa(s,\mu)$, and $\varrho_{2,n}(\varkappa)\rightarrow \varrho_{2}(\varkappa)$ as $n\rightarrow\infty$, the following estimate holds: 
	 \begin{equation*}\label{nonsmooth:inew:phi_s_kapeps_phi_kappa}
	 \phi\pairEkapeps+\frac{1}{2\varkappa^2}\big[|s-\letterEkapeps{t}|^2+W_2^2(\letterEkapeps{\nu},\mu)\big] \leq \phi_\varkappa(s,\mu)+\varepsilon \varrho_{2}(\varkappa).
	 \end{equation*} The conclusion of lemma now directly follows from this and~\eqref{nonsmooth:ineq:phi_s_mu_prime}. 
\end{proof}

\begin{lemma}\label{nonsmooth:lm:varphi_varphi_varkappa} For each $(s,\mu)\in D$, the following holds: 
	\[|\phi(s,\mu)-\phi_\varkappa(s,\mu)|\leq \varrho_{3}(\varkappa),\] where $\varrho_{3}(\kappa)\rightarrow 0$ as $\kappa\rightarrow 0$.
\end{lemma}
\begin{proof} First notice that 
	\begin{equation}\label{nonsmooth:ineq:varphi_1}
		\phi_\varkappa(s,\mu)\leq \phi(s,\mu).
	\end{equation}
	Furthermore, there exists a sequence $\{(t_n,\nu_n)\}_{n=1}^\infty$ such that
	\[\phi_\varkappa(s,\mu)+\frac{1}{n}\geq \phi(t_n,\nu_n)+\frac{1}{2\varkappa^2}\big[|s-t_n|^2+W_2^2(\mu,\nu_n)\big].\] Using~\eqref{nonsmooth:ineq:varphi_1} and denoting $A_n\triangleq \big[|s-t_n|^2+W_2^2(\mu,\nu_n)\big]^{1/2}$, we obtain
	\begin{equation}\label{nonsmooth:ineq:varphi_kappa_2}
		A_n^2\leq 2\varkappa^2|\phi_\varkappa(s,\mu)-\phi(t_n,\nu_n)|+\frac{2\varkappa^2}{n}.
	\end{equation} Since $\phi$ is bounded on $D^{(1)}$ by the constant $c_0$ (see~\eqref{nonsmooth:intro:c_0}), we have that 
	\[A_n^2\leq 4\varkappa^2c_0+\frac{2\varkappa^2}{n}.\] Using the fact that $\phi\in \operatorname{UC}([0,T]\times D^{(1)})$ and estimate~\eqref{nonsmooth:ineq:varphi_kappa_2}, we conclude that 
	\[A_n\leq \varkappa\varrho_{3,n}'(\varkappa),\] where \[\varrho_{3,n}'(\kappa)\triangleq \sqrt{2}\Big(\omega\big(2\kappa \sqrt{c_0}+\sqrt{2}\kappa/\sqrt{n}\big)\Big)^{1/2}+\frac{\sqrt{2}}{n^{1/2}}.\]
	 Now notice that 
	\[\begin{split}
	|\phi(s,\mu)-\phi_\varkappa(s,\mu)|&=\lim_{n\rightarrow \infty}\Bigg|\phi(s,\mu)-\phi(t_n,\nu_n)-\frac{1}{2\varkappa^2}\big[|s-t_n|^2+W_2^2(\mu,\nu_n)\big]\Bigg|\\&\leq 
	\lim_{n\rightarrow\infty}\Bigg[|\phi(s,\mu)-\phi(t_n,\nu_n)|+\frac{1}{2\varkappa^2}\big[|s-t_n|^2+W_2^2(\mu,\nu_n)\big]\Bigg]\\&\leq \omega\Big(\big[|s-t_n|^2+W_2^2(\mu,\nu_n)\big]^{1/2}\Big)+\frac{1}{2\varkappa^2}\big[|s-t_n|^2+W_2^2(\mu,\nu_n)\big]\\ &\leq \omega\big(\varkappa \varrho_{3,n}'(\varkappa)\big)+\frac{1}{2}\varrho_{3,n}'(\varkappa).
	\end{split}\] Passing to the limit when $n\rightarrow\infty$, we obtain the conclusion of the lemma with 
	\[\varrho_{3}(\kappa)\triangleq \omega[\kappa \varrho_{3,\infty}'(\kappa)]+\frac{1}{2}\varrho_{3,\infty}'(\kappa).\]
	Here we denote
	\[\varrho_{3,\infty}'(\kappa)\triangleq \sqrt{2}\Big(\omega\big(2\kappa \sqrt{c_0}\big)\Big)^{1/2}.\]
\end{proof}

\section{Controlled continuity equation}\label{sect:control}

We impose the following conditions on the dynamics and the payoff functional introduced in~\eqref{intrdct:payoff:main} and~\eqref{intrdct:eq:dynamics}.
\begin{enumerate}[label=(C\arabic*)]
	\item\label{control:cond:U} $U$ is a metric compact;
	\item\label{control:cond:f_omega} there exists a continuous function $\omega_f:[0,+\infty)\rightarrow [0,+\infty)$ vanishing at zero such that, for every $s,r\in [0,T]$, $x\in\rd$, $m\in\prd$ and $u\in U$,
	\[|f(s,x,m,u)-f(r,x,m,u)|\leq \omega_f(|s-r|)(1+|x|+\varsigma(m));\]
	\item\label{control:cond:f_Lip} the function $f$ is uniformly Lipschitz continuous w.r.t.\ $x$ and $m$;
	\item\label{control:cond:L} the function $L:[0,T]\times \mathcal{P}(\rd)\times U\rightarrow \mathbb{R}$ is continuous and uniformly Lipschitz continuous w.r.t.\ the measure variable;
	\item the function $G\in\operatorname{UC}(\prd)$.
\end{enumerate}

In what following, $C_f$ and $C_L$ denote the Lipschitz constants for the functions $f$ and $L$ w.r.t.\ $x$ and $m$.

Notice that, from~\ref{control:cond:f_Lip}, it follows that there exists a constant $C_1>0$ such that, for every $t\in [0,T]$, $x\in\rd$, $m\in\prd$, $u\in U$,
\begin{equation}\label{control:ineq:sublinear}
	|f(t,x,m,u)|\leq C_1(1+|x|+\varsigma(m)).
\end{equation}

Now let us recall the definition of the solution of the continuity equation. 
\begin{definition}\label{def:prel:cont_eq_distr}
	For $s,r\in [0,T]$, $s<r$ and a measurable function $u(\cdot):[s,r]\rightarrow U$, a flow of probabilities $[s,r]\ni t\mapsto m_t\in \prd$ is a solution of continuity equation~\eqref{intrdct:eq:dynamics} on $[s,r]$ if, for every test function $\phi\in C_0^1((s,r)\times\rd)$, the following holds:
	\[\int_{s}^{r}\int_{\rd}\Big[\partial_t\phi(t,x)+\nabla\phi(t,x)f(t,x,m_t,u(t))\Big]m_t(dx)dt=0. \]
\end{definition}

Along with measurable controls, it is convenient to use a relaxed (generalized) controls. Recall \cite{Warga1972} that a relaxed control on $[s,r]$ is a measure on $[s,r]\times U$ whose marginal distribution on $[s,r]$ coincides with the Lebesgue measure. We denote the set of all relaxed controls on $[s,r]$ by $\mathcal{U}_{s,r}$, i.e., $\mathcal{U}_{s,r}\triangleq \Lambda([s,r],\lambda;U)$, where $\lambda$ denotes the Lebesgue measure. If $\xi\in\mathcal{U}_{s,r}$, the disintegration w.r.t.\ the Lebesgue measure gives the weakly measurable mapping $[s,r]\ni t\mapsto \xi(\cdot|t)\in\mathcal{P}(U)$. 

Recall that we consider on $\mathcal{U}_{s,r}$ the topology of narrow convergence. Since $U$ is compact, the space $\mathcal{U}_{s,r}$ is also compact.

The measurable controls are embedded into the space of relaxed control, i.e, a measurable control $u(\cdot):[s,r]\rightarrow U$ corresponds to the measure $\xi_{u(\cdot)}\in \mathcal{U}_{s,r}$ defined by the rule: for each $\phi\in C([s,r]\times U)$,
\[\int_{[s,r]\times U}\phi(t,u)\xi_{u(\cdot)}(d(t,u))=\int_s^r\phi(t,u(t))dt.\] Moreover, due to \cite{Warga1972}, for each $\xi\in\mathcal{U}_{s,r}$, there exists a sequence of measurable controls $\{u_n(\cdot)\}_{n=1}^\infty$ on $[s,r]$ such that $\{\xi_{u_n(\cdot)}\}_{n=1}^\infty$ converges to $\xi$.

For $s\leq \theta\leq r$, given $\xi^1\in\mathcal{U}_{s,\theta}$ and $\xi^2\in \mathcal{U}_{\theta,r}$, their concatenation $\xi^1\diamond_\theta\xi^2\in\mathcal{U}_{s,r}$ is defined via disintegration as follows:
\[(\xi^1\diamond_\theta\xi^2)(\cdot|t)\triangleq 
\begin{cases}
	\xi^1(\cdot|t), & t\in [s,\theta),\\
	\xi^2(\cdot|t), & t\in [\theta,r].
\end{cases} \]

Continuity equation~\eqref{intrdct:eq:dynamics} in this case takes formally the form 
\begin{equation}\label{control:eq:continuity_relexed}\partial_t m_t+\operatorname{div}\Bigg(\int_U f(t,x,m_t,u)\xi(du|t)\cdot m_t\Bigg)=0.\end{equation}
A solution of this continuity equation governed by the relaxed control $\xi$ is defined as follows.
\begin{definition}\label{def:prel:cont_eq_distr_relaxed}
	Let $s,r\in [0,T]$, $s<r$, $\xi\in\mathcal{U}_{s,r}$. A flow of probabilities $[s,r]\ni t\mapsto m_t$ is a solution of continuity equation~\eqref{control:eq:continuity_relexed} on the interval $[s,r]$ provided that, for each test function $\phi\in C_0^1((s,r)\times\rd)$, one has that 
	\[\int_{s}^{r}\int_{\rd}\int_U\Big[\partial_t\phi(t,x)+\nabla\phi(t,x)f(t,x,m_t,u)\Big]\xi(du|t)m_t(dx)dt=0. \]
\end{definition} Below, we denote the solution of the continuity equation on $[s,r]$ that is governed by the relaxed control $\xi$ and meets the initial condition $m(s)=\mu$ by $m_\cdot[s,\mu,\xi]$. Notice that $m_\cdot[s,\mu,\xi]\in C([s,r];\prd)$.

The seminal superposition principle gives an equivalent form of this definition. To formulate it, we first, given $s,r\in [0,T]$, $m_\cdot\in C([s,r];\prd)$, $\xi\in\mathcal{U}_{s,r}$ denote by 
$\operatorname{X}_{m_\cdot,\xi}^{s,t}$ a flow that is generated by the differential equation 
\begin{equation}\label{control:eq:particle}\frac{d}{dt}x(t)=\int_U f(t,x(t),m_\cdot,u)\xi(du|t),\end{equation} i.e., $\operatorname{X}_{m_\cdot,\xi}^{s,t}(y)$ is equal to $x(t)$, where $x(\cdot)$ is a solution~\eqref{control:eq:particle} satisfying $x(s)=y$. 

\begin{proposition}\label{prop:control:superposition} A flow of probabilities $m_\cdot\in C([s,r];\prd)$ is a solution of the continuity equation on $[s,r]$ governed by a relaxed control $\xi\in\mathcal{U}_{s,r}$ if and only if, for every $t\in [s,r]$, 
	\[m_t=X^{s,t}_{m_\cdot,\xi}\sharp m_s.\]
\end{proposition} This statement directly follows from the \cite[Theorem 8.2.1]{ambrosio}

Notice that, given $s\leq \theta\leq r$, $\mu\in\prd$, $\xi^1\in\mathcal{U}_{s,\theta}$ and $\xi^2\in \mathcal{U}_{\theta,r}$, one has that, for $t\in [s,\theta]$ 
\begin{equation}\label{control:equality:conc_m_first}
	m_t[s,\mu,\xi^1\diamond_\theta\xi^2]=m_t[s,\mu,\xi^1],
\end{equation} while, for $t\in [\theta,r]$,
\begin{equation}\label{control:equality:conc_second}
	m_t[s,\mu,\xi^1\diamond_\theta\xi^2]=m_t[\theta,m_\theta[s,\mu,\xi^1]],\xi^2).
\end{equation}

The sequence $\{m_\cdot[s,\mu_n,\xi_n]\}_{n=1}^\infty$ converges to $m_\cdot[s,\mu,\xi]$ in the case where $\{\mu_n\}$ and $\{\xi_n\}$ converge to $\mu$ and $\xi$ respectively. This property is formalized as follows.
\begin{proposition}\label{control:prop:convergence} Let 
	\begin{itemize}
		\item $\{\mu_n\}_{n=1}^\infty\subset \prd$ converge to $\mu\in\prd$;
		\item $s,r\in [0,T]$, $s<r$;
		\item $\{\xi_n\}_{n=1}^\infty\subset\mathcal{U}_{s,r}$ converge to $\xi\in\mathcal{U}_{s,r}$.
	\end{itemize} Then $\{m_\cdot[s,\mu_n,\xi_n]\}_{n=1}^\infty$ converges to 
	$m_\cdot[s,\mu,\xi]$.
\end{proposition}
This statement is proved in \cite[Lemma 6.1]{Averboukh23}.

Now, we reformulate the optimal control problem~\eqref{intrdct:payoff:main},~\eqref{intrdct:eq:dynamics} using the relaxed controls. Given $s_*\in [0,T]$, $\mu\in \prd$, we consider the following optimal control problem:
\begin{equation*}\label{problem:control:criterion}
	\text{minimize }J[s_*,\mu,\xi]\triangleq \int_{[s_*,T]\times U}L(t,m_t[s_*,\mu_*,\xi],u)\xi(d(t,u))+G(m_T[s_*,\mu_*,\xi])
\end{equation*} over the set $\xi\in\mathcal{U}_{s_*,T}$.

Thus, the value function of optimal control problem~\eqref{intrdct:payoff:main},~\eqref{intrdct:eq:dynamics} is defined by the rule:
\[\operatorname{Val}(s_*,\mu_*)\triangleq \inf_{\xi\in \mathcal{U}_{s_*,T}}J[s_*,\mu_*,\xi].\]

Propositions~\ref{control:prop:convergence} and \ref{app:prop:bounds} immediately imply the following.
\begin{theorem}\label{control:th:existence} There exists a relaxed control $\xi^*\in\mathcal{U}_{s_*,T}$ such that 
	\[J[s_*,\mu_*,\xi^*]= \operatorname{Val}(s_*,\mu_*).\] Moreover, the function $\operatorname{Val}$ is uniformly continuous on each bounded subset of $[0,T]\times\prd$.
\end{theorem}

We complete this with the dynamic programming principle.
\begin{theorem}\label{control:th:dpp} Let $s_*\in [0,T]$, $\mu_*\in\prd$, $\theta\in [s_*,T]$. Then,
	\[\operatorname{Val}(s_*,\mu_*)=\min_{\xi\in \mathcal{U}_{s_*,\theta}}\operatorname{Val}(\theta,m(\theta,s_*,\mu_*,\xi)).\]
\end{theorem}
The proof follows from standard arguments and equalities~\eqref{control:equality:conc_m_first},~\eqref{control:equality:conc_second}; thus, we omit it.

\section{Characterization of the value function}\label{sect:Bellman}
In this section, we show that the value function is a viscosity solution of the corresponding Bellman equation that is a Hamilton-Jacobi PDE in the space of measures. As we mentioned in the Introduction, this result refines that of~\cite{Badreddine2021}, where only compactly supported measures were considered. Following the approach of~\cite{Jimenez2020,Jimenez2023}, we employ $\varepsilon$-sub- and superdifferentials.

We assign to the original control problem~\eqref{intrdct:payoff:main},~\eqref{intrdct:eq:dynamics} the Hamiltonian $H$ that is defined for $s\in [0,T]$, $\mu\in\prd$, $p\in \LTwoS{\mu}$ by the rule:
\begin{equation}\label{Bellman:intro:Hamiltonian}H(s,\mu,p)\triangleq \min_{u\in U}\Bigg[\int_{\rd}p(x) f(s,x,\mu,u)\mu(dx)+L(s,\mu,u)\Bigg].\end{equation}

The Bellman equation that corresponds to problem~\eqref{intrdct:payoff:main},~\eqref{intrdct:eq:dynamics} takes the form:
\begin{equation}\label{Bellman:eq:HJ}
	\partial_t\varphi+H(t,\mu,\nabla_\mu\varphi)=0. \end{equation} This equation is endowed with the boundary condition
\begin{equation}\label{Bellman:cond:boundary}
	\varphi(T,\mu)=G(\mu).
\end{equation}
In~\eqref{Bellman:eq:HJ}, we formally use 	$\nabla_\mu\varphi$ for the derivative of the function $\varphi$ w.r.t.\ the measure variable. We will proceed as in~\cite{Crandall1985} and introduce strict viscosity solutions.

\begin{definition}\label{Bellman:def:viscosity_solution} We say that a function $\varphi\in\operatorname{USCB}([0,T]\times\prd)$ is a strict viscosity subsolution of~\eqref{Bellman:eq:HJ} if, given a bounded subset $D\subset \prd$, one can find a constant $C(D)\geq 0$ such that, for every $s\in (0,T)$, $\mu\in D$, $\varepsilon>0$, $(a,p)\in \partial_{D,\varepsilon}^+\varphi(s,\mu)$,
	\[a+H(t,\mu,p)\geq - C(D)\varepsilon.\] Analogously, a function $\varphi\in\operatorname{LSCB}([0,T]\times\prd)$ is a strict viscosity supersolution of~\eqref{Bellman:eq:HJ} provided that, given a bounded subset $D$ of $\prd$, there exists a constant $C(D)\geq 0$ such that 
	\[a+H(t,\mu,p)\leq C(D)\varepsilon\] whenever $s\in (0,T)$, $\mu\in D$, $\varepsilon>0$, $(a,p)\in \partial_{D,\varepsilon}^-\varphi(t,\mu)$.
	
	Finally, a continuous function $\varphi:[0,T]\times\prd\rightarrow\mathbb{R}$ is called a strict viscosity solution of~\eqref{Bellman:eq:HJ} if it is a strict viscosity subsolution and a strict viscosity supersolution simultaneously.
\end{definition}
\begin{remark}
We say that a function $\varphi\in\operatorname{USCB}([0,T]\times\prd)$ is a viscosity subsolution if \[a+H(t,\mu,p)\geq 0\] for every $(s,\mu)\in (0,T)\times\prd$ and $(a,p)\in \partial_{D,0}^+\phi(s,\mu)$. 

Similarly,
a function $\varphi\in\operatorname{LSCB}([0,T]\times\prd)$ is a viscosity supersolution provided that \[a+H(t,\mu,p)\geq 0\] for every $(s,\mu)\in (0,T)\times\prd$ and $(a,p)\in \partial_{D,0}^-\phi(s,\mu)$.

 Since 
	\[\partial_{D,0}^+\phi(s,\mu)=\bigcap_{\varepsilon>0}\partial_{D,\varepsilon}^+\phi(s,\mu),\] each strict viscosity subsolution is viscosity subsolution. The same concerns the supersolutions.
\end{remark}

\begin{theorem}\label{Bellman:th:characterization} The value function of optimal control~\eqref{intrdct:payoff:main},~\eqref{intrdct:eq:dynamics} is a strict viscosity solution of Bellman equation~\eqref{Bellman:eq:HJ} and satisfies the boundary condition \[\operatorname{Val}(T,\mu)=G(\mu).\]
\end{theorem}
\begin{proof} 
	First, notice that the boundary condition is obvious.
	
	Furthermore, we prove that $\operatorname{Val}$ is a subsolution of~\eqref{Bellman:eq:HJ}.
	
	To this end, we fix a bounded set $D\subset\prd$, $(s,\mu)\in (0,T)\times D$, $\varepsilon>0$ and $(a,p)\in \partial^+_{D,\varepsilon}\operatorname{Val}(s,\mu)$. Furthermore, we choose $\hat{u}\in U$ arbitrarily. We let \[\zeta\triangleq \delta_{\hat{u}},\, v(y)\triangleq f(s,y,\mu,\hat{u})=\int_U f(s,y,\mu,u)\zeta(du).\] Additionally, we denote $m_\cdot\triangleq m_\cdot[s,\mu,\xi]$, where $\xi(\cdot|t)\triangleq \zeta(\cdot)$ whenever $t\in [s,s+r]$, $r>s$, and, for each $h\in (0,r-s]$, put \[v^h\triangleq \frac{1}{h}\int_s^{s+h}f(t,X^{s,t}_{m_\cdot},m_t,u)\zeta(du).\] Proposition~\ref{app:prop:v_zeta} gives that $\|v^h-v\|_{\LTwo{\mu}}\rightarrow 0$ as $h\rightarrow 0$. Thus, we use in the definition of the directional superdifferential (see Definition~\ref{nonsmooth:def:directional:diffe}) $\theta'=\theta=1$ and obtain, for each $\kappa>0$ and sufficiently small $h$
	\begin{equation}\label{Bellman:eq:super_diff_Val}
		h(a+\langle p,v^h\rangle_\mu)+\kappa h+h\varepsilon\|(1,v^h)\|_{\mathbb{R}\times\LTwo{\mu}}\geq \operatorname{Val}(s+h,(\operatorname{Id}+hv^h)\sharp\mu)-\operatorname{Val}(s,\mu). \end{equation}
	The dynamic programming principle (see Theorem~\ref{control:th:dpp}) implies that 
	\begin{equation}\label{Bellman:ineq:DPP_u}\operatorname{Val}(s,\mu)\leq \operatorname{Val}(s+h,m_{s+h})+\int_s^{s+h} L(t,m_t,\hat{u})dt.\end{equation} Moreover, by construction of the velocity fields $v^h$, we have that 
	\[m_{s+h}=(\operatorname{Id}+h v^h)\sharp\mu.\]
	From this,~\eqref{Bellman:ineq:DPP_u} and~\eqref{Bellman:eq:super_diff_Val}, we have that 
	\[h(a+\langle p,v^h\rangle_\mu)+\kappa h+h\varepsilon\|(1,v^h)\|_{\mathbb{R}\times\LTwo{\mu}}+ \int_s^{s+h} L(t,m_t,\hat{u})dt\geq 0.\] Dividing both parts on $h$, passing to the limit as $h\rightarrow0$ and taking into account Proposition~\ref{app:prop:L_limit_zeta} as well as convergence of $v^h$ to $v$, we conclude that 
	\begin{equation}\label{Bellman:ineq:preBellman_geq}a+\langle p,v\rangle_\mu+L(t,\mu,\hat{u})\geq -\kappa-\varepsilon\|(1,v)\|_{\mathbb{R}\times\LTwo{\mu}}.\end{equation} Now recall that $v(x)=f(s,x,\mu,u)$. Using~\eqref{control:ineq:sublinear}, we conclude that 
	\[\|(1,v)\|_{\mathbb{R}\times\LTwo{\mu}}\leq C(D),\]
	where
	\begin{equation}\label{Bellman:intro:C_D}
		C(D)\triangleq \sup_{m\in D}\big(1+(C_1)^2(1+2\varsigma(m))^2\big)^{1/2}.
	\end{equation}
	 Since $\hat{u}\in U$ is arbitrary, inequality~\eqref{Bellman:ineq:preBellman_geq} implies
	\[a+H(s,\mu,p)\geq -C(D)\varepsilon.\] Thus, $\operatorname{Val}$ is a subsolution of~\eqref{Bellman:eq:HJ}. 
	
	
		To prove the fact that $\operatorname{Val}$ is a supersolution of ~\eqref{Bellman:eq:HJ}, we, as above, fix a bounded set $D\subset\prd$ $s,\mu\in (0,T)\times D$, $\varepsilon>0$ and $(a,p)\in\partial_{D,\varepsilon}^-\operatorname{Val}(s,\mu)$. For each $h>0$, due to the dynamic programming principle there exists a relaxed control $\xi^h\in\mathcal{U}_{s,s+h}$ such that 
	\begin{equation}\label{Bellman:equality:Val_optimal_h}
		\operatorname{Val}(s,\mu)= \operatorname{Val}(s+h,m_{s+h}[s,\mu,\xi^h])+\int_s^{s+h}\int_U L(t,m_t[s,\mu,\xi^h],u)\xi^h(du|t)dt.
	\end{equation} One can define the averaging of the measure on $\xi^h$ over time interval $[s,s+h]$, i.e., we put, for each $\psi\in C(U)$,
	\[\int_U\psi(u)\zeta^h(du)\triangleq \frac{1}{h}\int_s^{s+h}\int_U\psi(u)\xi^h(du|t)dt.\] Notice that $\zeta^h\in\mathcal{P}(U)$. Since $\mathcal{P}(U)$ is compact, there exist a sequence $\{h_n\}_{n=1}^\infty$ tending to zero and a probability $\zeta\in\mathcal{P}(U)$ such that $\zeta^{(n)}\triangleq \zeta^{h_n}\rightarrow \zeta$ as $n\rightarrow \infty$. We use the notation of the Appendix and put, $\xi^{(n)}\triangleq \xi^{h_n}$, $m^{(n)}_\cdot\triangleq m_\cdot[s,\mu,\xi^{(n)}]$, \[v^{(n)}(y)\triangleq \frac{1}{h_n}\int_s^{s+h_n}f(t,X^{s,t}_{m^{(n)}_\cdot}(y),m^{(n)}_t,u)\xi^{(n)}(du|t)dt,\]
	\[v(y)\triangleq \int_U f(s,y,\mu,u)\zeta(du).\] Notice that $ m^{(n)}_{s+h_n}= (\operatorname{Id}+h_n v^{(n)})\sharp\mu.$
	Proposition~\ref{app:prop:v_xi_zeta_n} gives that $\|v^{(n)}-v\|_{\LTwo{\mu}}\rightarrow 0$ as $n\rightarrow \infty$. Thus, since $(a,p)\in\partial_{D,\varepsilon}^-\operatorname{Val}(s,\mu)$, given $\kappa>0$, one can find $N$ such that, for every $n>N$, the following inequality holds:
	\begin{equation}\label{Bellman:ineq:superdiff_Val}
	\begin{split}
		h_n(a+\langle p,v^{(n)}\rangle_\mu)-&\kappa h_n-\varepsilon\|(1,v^{(n)})\|_{\mathbb{R}\times\LTwo{\mu}}\\&\leq \operatorname{Val}(s+h_n,(\operatorname{Id}+h_n v^{(n)})\sharp\mu)-\operatorname{Val}(s,\mu).\end{split}
	\end{equation} This, the fact that $m_{s+h_n}[s,\mu,\xi^{h_n}]=m^{(n)}_{s+h_n}= (\operatorname{Id}+h_n v^{(n)})\sharp\mu$ and equality~\eqref{Bellman:equality:Val_optimal_h} imply
	\[\begin{split}
	h_n(a+\langle p,v^{(n)}\rangle_\mu)-&\kappa h_n-h_n\varepsilon\|(1,v^{(n)})\|_{\mathbb{R}\times\LTwo{\mu}}\\&\leq -\int_s^{s+h_n}\int_U L(t,m_t^{(n)},u)\xi^{(n)}(du|t)dt.\end{split}\] Now, we divide both parts on $h_n$, pass to the limit when $n\rightarrow\infty$ and use Propositions~\ref{app:prop:v_xi_zeta_n} and~\ref{app:prop:L_xi_zeta}. Hence, 
	\[a+\langle p,v\rangle_\mu+\int_U L(s,\mu,u)\zeta(du)\leq \kappa+\varepsilon\|(1,v)\|_{\mathbb{R}\times\LTwo{\mu}}.\] Now recall that $\|(1,v)\|_{\mathbb{R}\times\LTwo{\mu}}\leq C(D)$. Since $\kappa$ is an arbitrary positive number, due to the Tonelli theorem, we conclude that 
	\[a+\int_U\bigg[\int_{\rd}p(y)f(s,y,\mu,u)\mu(dy)+L(s,\mu,u)\bigg]\zeta(du)\leq C(D)\varepsilon.\] This directly gives the inequality
	\[a+H(s,\mu,p)\leq C(D)\varepsilon.\] The latter means that $\operatorname{Val}$ is a supersolution of~\eqref{Bellman:eq:HJ}.
	
\end{proof}

\section{Feedback strategies and upper bound of the value function}\label{sect:feedback}
First, let us recall the Krasovskii-Subbotin approach to the definition of the feedback strategy. It implies that a feedback strategy is an arbitrary function $\mathfrak{u}:[0,T]\times\prd\rightarrow U$. To construct a motion, we fix an initial time $s_*$, initial distribution $\mu_*$ and a partition $\Delta=(s_i)_{i=0}^n$ of the time interval $[s_*,T]$. The motion $m_\cdot$ generated by $\mathfrak{u}$ is defined as follows. First, we put $m_{s_0}\triangleq \mu_*$. Then, for $t\in [s_i,s_{i+1})$, we put 
\begin{enumerate}[label=(M\arabic*)]
	\item\label{feedback:cond:control} $\xi(\cdot|t)\triangleq \delta_{\mathfrak{u}[s_i,m_{s_i}]}(\cdot)$;
	\item\label{feedback:cond:motion} $m_t\triangleq m_t[s_i,m_{s_i},\xi^i]$.
\end{enumerate} We call a pair $(m_\cdot,\xi)$ defined by rules~\ref{feedback:cond:control},~\ref{feedback:cond:motion} a control process generated by the strategy $\mathfrak{u}$ and initial position $(s_*,\mu_*)$. We denote it by $\mathcal{X}[s_*,\mu_*,\mathfrak{u},\Delta]$. 

Notice that if $(m_\cdot,\xi)=\mathcal{X}[s_*,\mu_*,\mathfrak{u},\Delta]$, then 
\[m_\cdot=m_\cdot[s_*,\mu_*,\xi].\] Thus, 
\[\begin{split}
\mathfrak{J}[s_*,\mu_*,\mathfrak{u},\Delta]\triangleq J[s_*,\mu_*,\xi]&= G(m_T)+\int_{[s_*,T]\times U}L(t,m_t,u)\xi(d(t,u))\\&=G(m_T)+\sum_{i=0}^{n-1}\int_{s_i}^{s_{i+1}}L(t,m_t,\mathfrak{u}[s_i,m_{s_i}])dt.\end{split}\]

Now, we define the feedback strategy $\letterEkapeps{\mathfrak{u}}$. This strategy realizes the concept of proximal aiming and is an analog of one proposed in \cite{Clarke1999}. First, we fix a bounded subset $D\subset\prd$ and a function $\varphi\in \operatorname{UC}([0,T]\times\prd)$. Given $(s,\mu)\in [0,T]\times D$, we choose $\pairkapeps{s}{\mu}$ satisfying~\eqref{nonsmooth:intro_ineq:s_mu_t_nu_alpha_eps}. Further, let $\letterkapeps{\pi}{s}{\mu}$ be an optimal plan between $\mu$ and $\letterkapeps{t}{s}{\mu}$. Now,
\[\letterkapeps{a}{s}{\mu}\triangleq s-\letterkapeps{t}{s}{\mu},\ \ \letterkapeps{\gamma}{s}{\mu}\triangleq (\operatorname{p}^2,\varkappa^{-2}(\operatorname{p}^1-\operatorname{p}^2))\sharp\letterkapeps{\pi}{s}{\mu}.\]
 Due to Lemma~\ref{nonsmooth:prop:diff_prox_directional}, $(\letterkapeps{a}{s}{\mu},\mathscr{b}[\letterkapeps{\gamma}{s}{\mu}])\in \partial_{P,\varepsilon}^-\varphi(s,\mu)$ when $(s,\mu)\in (0,T)\times D$ while $\varkappa$ and $\varepsilon$ are sufficiently small.
 
Furthermore, let, for each $(s,\mu)$, $\letterkapeps{\mathfrak{u}}{s}{\mu}$ satisfy
\begin{equation}\label{feedback:intro:u}
\begin{split}
	\letterkapeps{\mathfrak{u}}{s}{\mu}\in\underset{u\in U}{\operatorname{Argmax}}\Bigg[\int_{\rd}\mathscr{b}[\letterkapeps{\gamma}{s}{\mu}](x)f(\letterkapeps{t}{s}{\mu},x&,\letterkapeps{\nu}{s}{\mu},u)\mu(dx)\\&+L(\letterkapeps{t}{s}{\mu},\letterkapeps{\nu}{s}{\mu},u)\Bigg].\end{split}
\end{equation}

Now let $D_0\subset \prd$. By Proposition~\ref{app:prop:bounds}, there exists a domain $D$ such that each motion starting at $ D_0$ does not leave the set 
\begin{equation}\label{feedback:intro:D}
	D\triangleq \Bigg\{\nu:\, \varsigma(\nu)\leq c_1(1+\sup_{\mu\in D_0}\varsigma(\mu))\Bigg\}.
\end{equation} As in Section~\ref{sect:nonsmooth}, we put \begin{equation}\label{feedback:intro:D_1}D^{(1)}\triangleq \big\{\mu\in\prd:\, \operatorname{dist}(\mu,D)\leq 1\big\}.\end{equation}
In what follows, given a precision parameter $\eta>0$, we will consider a set 
$\mathscr{S}(\eta)$ consisting of 4-tuples $(\varkappa,\varepsilon,\alpha_*,\alpha^*)$ satisfying 
$\alpha_*\leq\alpha^*$. The first two parameters, $\varkappa$ and $\varepsilon$, are used in the definition of the strategy $\letterEkapeps{\mathfrak{u}}$, while $\alpha_*$ and $\alpha^*$
determine the lower and upper bounds of an interval between control corrections. Each such set is referred to as a parameter complex.
\begin{theorem}\label{feedback:th:u_kapeps}
	Assume that $D_0$ is a bounded domain, $D$ and $D^{(1)}$ are defined by~\eqref{feedback:intro:D},\eqref{feedback:intro:D_1}, while $\varphi\in \operatorname{UC}([0,T]\times D^{(1)})$ is such that 
	\begin{itemize}
		\item $\varphi(T,\mu)\geq G(\mu)$;
		\item there exist constants $C>0$ and $\varepsilon_0>0$ satisfying the following condition: for every $(s,\mu)\in [0,T]\times D^{(1)}$, $\varepsilon\in (0,\varepsilon_0)$, $(a,\gamma)\in \partial_{P,\varepsilon}^-\varphi(s,\mu)$ with $\mathscr{b}[\gamma]\in\operatorname{dis}^-(\mu)$, 
		\[a+H(s,\mu,\mathscr{b}[\gamma])\leq C\varepsilon.\]
	\end{itemize} Then, given a precision $\eta>0$, one can find a parameter complex $\mathscr{S}(\eta)$ such that, for every $(s_*,\mu_*)\in [0,T]\times D_0$, $(\varkappa,\varepsilon,\alpha_*,\alpha^*)\in\mathscr{S}(\eta)$ and a partition $\Delta=\{s_i\}_{i=0}^n$ satisfying \[\alpha_*\leq s_{i+1}-s_i\leq \alpha^*\] one has that
	\begin{equation}\label{feedback:ineq:J_u}
		\mathfrak{J}[s_*,\mu_*,\letterEkapeps{\mathfrak{u}},\Delta]\leq \varphi(s_*,\mu_*)+\eta.
	\end{equation} 
\end{theorem}
\begin{proof} First, we fix a precision $\eta>0$ and a 4-tuple $(\varkappa,\varepsilon,\alpha_*,\alpha^*)$ such that $\varrho_{1}(\varkappa)<1$, $\alpha_*\leq\alpha^*$. Let $\Delta$ be a partition of $[s_*,T]$ satisfying
	\[\alpha_*\leq \min_{i\in \{0,\ldots,n-1\}}(s_{i+1}-s_i),\ \ \alpha^*\geq \max_{i\in \{0,\ldots,n-1\}}(s_{i+1}-s_i). \]
	Furthermore, for $(m_\cdot,\xi)=\mathcal{X}[s_*,\mu_*,\mathfrak{u},\Delta]$, we denote
		\begin{itemize}
		\item $\mu_i\triangleq m_{s_i}$;
		\item $t_i\triangleq \letterkapeps{t}{s_i}{\mu_i}$;
		\item $\nu_i\triangleq \letterkapeps{\nu}{s_i}{\mu_i}$;
		\item $u_i\triangleq \letterEkapeps{\mathfrak{u}}[s_i,\mu_i]$;
		\item $\pi_i\triangleq \letterkapeps{\pi}{s_i}{\mu_i}$;
		\item $a_i\triangleq \letterkapeps{a}{s_i}{\mu_i}=(s_i-\letterkapeps{t}{s_i}{\mu_i})$;
		\item $\gamma_i\triangleq \letterkapeps{\gamma}{s_i}{\mu_i}$.
	\end{itemize}
	We will also use the symbol $\omega(\cdot)$ for the modulus of continuity of the function $\varphi$ on $D^{(1)}$ (see~\eqref{nonsmooth:intro:modulus} for details). 
	
	Let $k$ be the maximal number such that $s_{k-1}\leq \varrho_{1}(\varkappa)$. Furthermore, we choose $l$ to be a minimal number satisfying $T-s_l\leq \varrho_{1}(\varkappa)$. Thus, from Corollary~\ref{nonsmooth:corollary:bound_2},
	 we have that $(t_i,\nu_i)\in (0,T)\times \operatorname{int}(D^{(1)})$ for every $i\in \{k,\ldots,l-1\}$.
	 
	 Furthermore, one can define
	\[ \varrho_{4}'(\kappa)\triangleq \varrho_{3}(\kappa)+ \omega(C_2\varrho_{1}(\kappa))\] where 
	\[C_2=1\vee \bigg[c_2\Big(1+\sup_{\mu\in D_0}\varsigma(\mu)\Big)\bigg],\] while
	$c_2$ is defined in Proposition~\ref{app:prop:bounds}. Notice that due to Proposition~\ref{app:prop:bounds},
	\[W_2(m_T,\mu_l)\leq C_2\cdot (T-s_l).\] 
	Hence, we use Lemma~\ref{nonsmooth:lm:varphi_varphi_varkappa}, and obtain
	\begin{equation}\label{feedback:ineq:varphi_kappa_s_l_T}
		|\varphi_\varkappa(s_l,\mu_l)- \varphi(T,m_T)|\leq \varrho_{4}'(\varkappa).
	\end{equation} 
	Due to Proposition~\ref{app:prop:bounds}, we have that 
	\[C_3\triangleq \sup\{|L(t,\mu,u)|:\  t\in [0,T],\,\mu\in D,\, u\in U\}<\infty.\] Therefore,
	\begin{equation}\label{feedback:ineq:L_s_l_T}
		\Bigg|\int_{s_l}^T\int_U L(t,m_t,\xi)\xi(du|t)dt\Bigg|\leq C_3\varrho_{1}(\varkappa).
	\end{equation} 
	
	Furthermore, we put 
	\begin{equation}\label{feedback:intro:varrho_5}\varrho_{5}'(\kappa,\alpha)\triangleq \omega(C_2\varrho_{1}(\kappa)+C_2\alpha).\end{equation}
	 Using Lemma~\ref{nonsmooth:lm:varphi_varphi_varkappa}, we arrive at the estimate
	\begin{equation}\label{feedback:ineq:varphi_kappa_s_k_0}
		|\varphi(s_0,\mu_0)- \varphi(s_k,\mu_k)|\leq \varrho_{5}'(\varkappa,\alpha^*).
	\end{equation} 
	Similarly to \eqref{feedback:ineq:L_s_l_T}, we have that	\begin{equation}\label{feedback:ineq:L_s_0_s_k}
		\Bigg|\int_{s_0}^{s_k}\int_U L(t,m_t,\xi)\xi(du|t)dt\Bigg|\leq C_3(\varrho_{1}(\varkappa)+\alpha^*).
	\end{equation}

	Now, let $i\in \{k,\ldots,l-1\}$. In this case, one can introduce the function
	\[v_i(x)\triangleq \frac{1}{s_{i+1}-s_i}\int_{s_{i}}^{s_{+1}}f(t,X_{m_\cdot}^{s_i,t}(x),m_t,u_i)dt.\] 
	
	 Notice that 
	\[\mu_{i+1}=(\operatorname{Id}+(s_{i+1}-s_i)v_i)\sharp \mu_i\]
	 From Lemma~\ref{nonsmooth:lm:shift}, we have that 
	 \begin{equation}\label{feedback:ineq:varphi_i_first}
	 	 \begin{split}
	 		\varphi_\varkappa(s_{i+1}&,\mu_{i+1})\leq \varphi_\varkappa(s_i,\mu_i)\\&+ a_i(s_{i+1}-s_i)+\frac{1}{\varkappa^2}\int_{\rd\times\rd}(x-y)(s_{i+1}-s_i)v_i(x)\pi_i(d(x,y))\\&+\frac{C_4}{2\varkappa^2}(s_{i+1}-s_i)^2+\varepsilon\varrho_{2}(\varkappa).
	 	\end{split}
	 \end{equation} Here $C_4$ is a constant certainly dependent on $D$ and, thus, on $D_0$. One can use conditions~\ref{control:cond:f_omega} and~\ref{control:cond:f_Lip} and Proposition~\ref{app:prop:bounds}. Thus,
	 \[
	 \begin{split}
	 	|f(t,X_{m_\cdot}^{s_i,t}(&x),m_t,u_i)-f(t_i,y,\mu_i,u_i)|\\\leq &C_5\omega_f(t-s_i)+C_5\omega_f(|s_i-t_i|)+C_6(t-s_i)+C_f|x-y|+C_fW_2(\mu_i,\nu_i).
	 \end{split}
	 \] Here, $C_5$ and $C_6$ are constants determined by $D_0$, while $C_f$ is the Lipschitz constant for the function $f$.
	 Hence, 
	 \[\begin{split}
	 \Bigg|\int_{\rd\times\rd}(x-&y)v_i(x)\pi_i(d(x,y))-\int_{\rd\times\rds}
	 (x-y^*)f(t_i,x,\nu_i,u_i)\pi_i(d(x,y^*))
	 \Bigg|\\\leq &\big[C_5\omega_f(t-s_i)+C_5\omega_f(|s_i-t_i|)+C_6(t-s_i)+2C_fW_2(\mu_i,\nu_i)\big]W_2(\mu_i,\nu_i).
	 \end{split}\]
	 
	From this estimate, inequality~\eqref{feedback:ineq:varphi_i_first} and Corollary~\ref{nonsmooth:corollary:bound_2}, we obtain
		 \begin{equation}\label{feedback:ineq:varphi_i_second}
		\begin{split}
			\varphi_\varkappa(s_{i+1}&,\mu_{i+1})\leq \varphi_\varkappa(s_i,\mu_i)\\&+ a_i(s_{i+1}-s_i)+(s_{i+1}-s_i)\int_{\rd}\mathscr{b}[\gamma_i](y)f(t_i,x,\nu_i,u_i)\mu_i(dy)\\&+(s_{i+1}-s_i)\varrho_{6}'(\varkappa,s_{i+1}-s_i)+(s_{i+1}-s_i)\varrho_{7}'(\varkappa)+\varepsilon\varrho_{2}(\varkappa).
		\end{split}
	\end{equation} Here we denote
	\[\varrho_{6}'(\kappa,\alpha)\triangleq \frac{C_4}{2\varkappa^2}\alpha+\Big[C_5\omega_f(\alpha)+C_6\alpha\Big]\varrho_{1}(\kappa),\]
	 \[\varrho_{7}'(\kappa)\triangleq \big[C_5\omega_f(\varrho_{1}(\kappa))+2C_f\varrho_{1}(\kappa)\big]\varrho_{1}(\kappa).\] Notice that 
	 $\varrho_{6}'(\kappa,\alpha)\rightarrow 0$ as $\alpha\rightarrow 0$, whereas $\varrho_{7}'(\kappa)\rightarrow 0$ as $\kappa\rightarrow 0$.
	 
	 Now let us consider the running cost. Due to condition~\ref{control:cond:L}, we have that, for $t\in [s_i,s_{i+1}]$,
	 \[
	 \begin{split}
	 |L(t,&m_t,u_i)-L(t_i,\nu_i,u_i)|\\&\leq \omega_{L}(t-s_i)+\omega_{L}(|s_i-t_i|)+C_LW_2(m_t,\mu_i)+C_LW_2(\mu_i,\nu_i).	
	 \end{split}
	 \] Here $C_L$ is a Lipschitz constant of the function $L$ w.r.t.\ the measure variable, while
	 \[\begin{split}
	 \omega_{L}(\alpha)\triangleq \sup\big\{|L(t'&,\nu,u)-L(t'',\nu,u)|:\\ &t',t''\in [0,T],\,|t'-t''|\leq \alpha,\, \nu\in D^{(1)},\, u\in U\big\}.\end{split}\] From this and~\eqref{feedback:ineq:varphi_i_second}, we conclude that 
	 \begin{equation}\label{feedback:ineq:varphi_i_third}
	 	\begin{split}
	 		\varphi_\varkappa(s_{i+1}&,\mu_{i+1})+\int_{s_i}^{s_{i+1}}L(t,m_t,u_i)dt\leq \varphi_\varkappa(s_i,\mu_i)\\&+ a_i(s_{i+1}-s_i)+(s_{i+1}-s_i)\int_{\rd}\mathscr{b}[\gamma_i](y)f(t_i,x,\nu_i,u_i)\mu_i(dy)\\&+(s_{i+1}-s_i)L(t_i,\mu_i,u_i)\\&+(s_{i+1}-s_i)\varrho_{8}'(\varkappa,s_{i+1}-s_i)+(s_{i+1}-s_i)\varrho_{9}'(\varkappa)+\varepsilon\varrho_{2}(\varkappa).
	 	\end{split}
	 \end{equation} Here, we used Proposition~\ref{app:prop:bounds} and denoted
	 	\[\varrho_{8}'(\kappa,\alpha)\triangleq \varrho_{6}'(\kappa,\alpha)+\omega_{L}(\alpha)+C_Lc_2\bigg(1+\sup_{\mu\in D_0}\varsigma(\mu)\bigg)\alpha,\]
	 \[\varrho_{9}'(\kappa)\triangleq \varrho_{7}'(\kappa)+\omega_{L}(\varrho_{1}(\kappa))+C_L\varrho_{1}(\kappa).\] As above 
	 $\varrho_{8}'(\kappa,\alpha)\rightarrow 0$ as $\alpha\rightarrow 0$. Simultaneously, $\varrho_{9}'(\kappa)\rightarrow 0$ as $\kappa\rightarrow 0$. 
	 
	 Furthermore, 
	 \[
	 	a_i+\int_{\rd}\mathscr{b}[\gamma_i](y)f(t_i,x,\nu_i,u_i)\mu_i(dy)+L(t_i,\mu_i,u_i)= a_i+H(t_i,\mu_i,\mathscr{b}[\gamma_i]).
	 \] Since $(a_i,\mathscr{b}[\gamma_i])\in \partial_{D,\varepsilon}^-\varphi(t_i,\mu_i)$, while $\mathscr{b}[\gamma_i]\in\operatorname{dis}^-(\mu_i)$, the assumption of the theorem gives that 
	 \[
	 a_i+\int_{\rd}\mathscr{b}[\gamma_i](y)f(t_i,x,\nu_i,u_i)\mu_i(dy)+L(t_i,\mu_i,u_i)\leq C\varepsilon.
	 \]
	 
	 Thus, summing inequalities~\eqref{feedback:ineq:varphi_i_third} for $i=k,\ldots,l-1$, we obtain
	 \[
	 \begin{split}
	 	\varphi_\varkappa(s_{l},\mu_{l})&+\sum_{i=k}^{l-1}\int_{s_i}^{s_{i+1}}L(t,m_t,u_i)dt\\\leq &\varphi_\varkappa(s_k,\mu_k)+TC\varepsilon+T\varrho_{8}'(\varkappa,\alpha^*)+T\varrho_{9}'(\varkappa)+n\varepsilon\varrho_{2}(\varkappa).
	 \end{split}
	 \]
	 Now we take into account inequalities~\eqref{feedback:ineq:varphi_kappa_s_l_T},~\eqref{feedback:ineq:L_s_l_T},\eqref{feedback:ineq:varphi_kappa_s_k_0},~\eqref{feedback:ineq:L_s_0_s_k}. It gives that 
	 \[
	 \begin{split}
	 	\varphi(T,m_T)&+\sum_{i=0}^{n-1}\int_{s_i}^{s_{i+1}}L(t,m_t,u_i)dt\\\leq &\varphi(s_0,\mu_0)+TC\varepsilon+T\varrho_{8}'(\varkappa,\alpha^*)+T\varrho_{9}'(\varkappa)+\frac{T}{\alpha_*}\varepsilon\varrho_{2}(\varkappa)\\&+\varrho_{4}' (\varkappa)+C_3\varrho_{1}(\varkappa)+\varrho_{5}'(\varkappa,\alpha^*)+C_3(\varrho_{1}(\varkappa)+\alpha^*).
	 \end{split}
	 \] Now recall that $\xi(t)=u_i$ whenever $t\in [s_i,s_{i+1})$. Therefore,
	 \begin{equation}\label{feedback:ineq:J_bound}
	 	\mathfrak{J}[s_*,\mu_*,\letterEkapeps{\mathfrak{u}},\Delta]\leq \varphi(s_*,\mu_*)+TC\varepsilon+\frac{T}{\alpha_*}\varepsilon\varrho_{2}(\varkappa)+\varrho_{10}'(\varkappa)+\varrho_{11}'(\varkappa,\alpha^*).
	 \end{equation} Here we denote
	 \[\varrho_{10}'(\kappa)\triangleq T\varrho_{9}'(\kappa)+\varrho_{4}' (\kappa)+C_3\varrho_{1}(\kappa)+C_3\varrho_{1}(\kappa)+\omega(C_2\varrho_{1}(\kappa)),\]
	 \[\varrho_{11}'(\kappa,\alpha)\triangleq T\varrho_{8}'(\kappa,\alpha)+\varrho_{5}'(\kappa,\alpha)-\omega(C_2\varrho_{1}(\kappa))+C_3\alpha.\] Notice that the functions $\varrho_{2}$, $\varrho_{10}'$ are monotonically increasing and $\varrho_{2}(\kappa),\varrho_{10}'(\kappa)\rightarrow 0$ as $\kappa\rightarrow 0$. At the same time, we recall the convergence $\varrho_{8}'(\kappa,\alpha)\rightarrow 0$ as $\alpha\rightarrow 0$. Moreover, due to the definition of $\varrho_{5}'$ (see~\eqref{feedback:intro:varrho_5}), for each $\kappa$, $\varrho_{5}(\kappa,\alpha)\rightarrow \omega(C_2\varrho_{1}(\kappa))$ as $\alpha\rightarrow 0$. This gives that $\varrho_{11}'(\kappa,\alpha)\rightarrow 0$ as $\alpha\rightarrow 0$. Additionally, the mapping $\alpha\mapsto \varrho_{11}'(\kappa,\alpha)$ is also monotonically increasing.
	 
	 Now, we find conditions $(\varkappa,\varepsilon,\alpha_*,\alpha^*)$ ensuring inequality~\eqref{feedback:ineq:J_u} and, thus, define the parameter complex $\mathscr{S}(\eta)$.
	 To this end, we let \begin{enumerate}
	 	\item 
	 	$\hat{\varkappa}(\eta)$ be such that $\varrho_{10}'(\varkappa)\leq \eta/3$ and $\varrho_{1}(\varkappa)<1$ for every $\varkappa\in (0,\hat{\varkappa}(\eta))$;
	 	\item $\alpha^0(\eta,\varkappa)$ be such that 
	 	$\varrho_{11}'(\varkappa,\alpha^*)\leq \eta/3$ for each $\alpha^*\leq \alpha^0(\eta,\varkappa)$;
	 	\item $\hat{\varepsilon}(\eta,\varkappa,\alpha^*)\in (0,\varepsilon_0)$ satisfy $TC\varepsilon+T\varepsilon^{1/2}\varrho_{2}(\varkappa)\leq \eta/3$ and $\varepsilon\leq\alpha^*$ whenever $\varepsilon\leq \hat{\varepsilon}(\eta,\varkappa,\alpha^*)$.
	 \end{enumerate} Combining this, we put
	 \[\begin{split}\mathscr{S}(\eta)\triangleq \Big\{(\varkappa,\varepsilon,\sqrt{\varepsilon},\alpha^*):\ \varkappa\in (0,\hat{\varkappa}(\eta)),\ \alpha^*\in (0,&\alpha^0(\varkappa,\eta)), \varepsilon\in (0,\hat{\varepsilon}(\eta,\varkappa,\alpha^*))\Big\}.\end{split} \] In particular, for every $(\varkappa,\varepsilon,\alpha_*,\alpha^*)\in\mathscr{S}(\eta)$, $\alpha_*\leq \alpha^*$. The very definition of the complex $\mathscr{S}(\eta)$ and estimate~\eqref{feedback:ineq:J_bound} give inequality~\eqref{feedback:ineq:J_u} for every $(\varkappa,\varepsilon,\alpha_*,\alpha^*)\in\mathscr{S}(\eta)$.
\end{proof} 

\begin{corollary}\label{feedback:corollary:bound} If $\varphi\in \operatorname{UC}([0,T]\times\prd)$ is a viscosity supersolution of~\eqref{Bellman:eq:HJ} satisfying $\varphi(T,\mu)\geq G(\mu)$, then, for every $(s_*,\mu_*)\in [0,T]\times\prd$, one has that 
	\[\operatorname{Val}(s_*,\mu_*)\leq \varphi(s_*,\mu_*).\]
\end{corollary}
\begin{proof}
	Theorem~\ref{feedback:th:u_kapeps} implies that, for every $\eta>0$, there exists $\xi_\eta\in\mathcal{U}_{s_*,T}$ such that 
	\[J[s_*,\mu_*,\xi_\eta]\leq \varphi(s_*,\mu_*)+\eta.\] 
Thus,
\[\operatorname{Val}(s_*,\mu_*)\leq \varphi(s_*,\mu_*)+\eta.\] Letting $\eta\rightarrow 0$, the corollary follows.
\end{proof}

\section{Lower bound of the value function}\label{sect:lb}
In this section, we work with a function $\psi$ and its upper Moreau-Yosida regularization that is defined by the $\sup$-convolution.

If $\psi\in \operatorname{UC}([0,T]\times\prd)$, then we denote $\varphi\triangleq -\psi$. 
Recall that, for every $(s,\mu)\in [0,T]\times\prd$, $\varepsilon>0$, 
\[\partial_{D,\varepsilon}^+\psi(s,\mu)=-\partial_{D,\varepsilon}^-\varphi(s,\mu).\] Analogously, one can introduce the proximal $\varepsilon$-superdifferential. In this case, $(\hat{a},\hat{\gamma})\in \partial_{P,\varepsilon}^+\psi(s,\mu)$ if and only if
$(-\hat{a},(\operatorname{p}^1,-\operatorname{p}^2)\sharp\hat{\gamma})\in \partial_{P,\varepsilon}^-(-\psi)(s,\mu)$. As in the case of subdifferentials, if $(\hat{a},\hat{\gamma})\in \partial_{P,\varepsilon}^+\psi(s,\mu)$,
\begin{equation}\label{lb:incl:subdiff}
	(\hat{a},\mathscr{b}[\hat{\gamma}])\in\partial_{D,\varepsilon}^+\psi(s,\mu).
\end{equation}

Let $D$ be a bounded subset of $\prd$ and let $D^{(1)}$ be its closed $1$-neighborhood. One can introduce the upper Moreau-Yosida regularization by the rule:
\[\psi^\varkappa(s,\mu)\triangleq \sup\Bigg\{\psi(t,\nu)-\frac{1}{2\varkappa^2}\big[|t-s|^2+W_2^2(\mu,\nu)\big]:\ \ t\in [0,T],\, \nu\in D^{(1)}\Bigg\}.\]
Notice that
\[\begin{split}\psi^\varkappa(s,\mu)&=-\inf\Bigg\{\varphi(t,\nu)+\frac{1}{2\varkappa^2}\big[|t-s|^2+W_2^2(\mu,\nu)\big]:\ \ t\in [0,T],\, \nu\in D^{(1)}\Bigg\}
\\&=-\varphi_\varkappa(s,\mu).
\end{split}\] Keeping the designation of Section~\ref{subsect:nonsmooth:MY}, we denote a pair satisfying~\eqref{nonsmooth:intro_ineq:s_mu_t_nu_alpha_eps},~\eqref{nonsmooth:intro_ineq:t_nu_alpha_eps} by $\pairkapeps{s}{\mu}$. Additionally, we fix an optimal plan  between $\mu$ and $\letterkapeps{\nu}{s}{\mu}$ denoted by $\letterkapeps{\pi}{s}{\mu}$. As in Section~\ref{subsect:nonsmooth:MY},
\[\letterkapeps{a}{s}{\mu}\triangleq s-\letterkapeps{t}{s}{\mu},\ \ \letterkapeps{\gamma}{s}{\mu}\triangleq (\operatorname{p}^2,\varkappa^{-2}(\operatorname{p}^1-\operatorname{p}^2))\sharp\letterkapeps{\pi}{s}{\mu}.\] If one let
\[\letterkapeps{\hat a}{s}{\mu}\triangleq \letterkapeps{t}{s}{\mu}- s,\]
\[\letterkapeps{\hat\gamma}{s}{\mu}\triangleq (\operatorname{p}^2,\varkappa^{-2}(\operatorname{p}^2-\operatorname{p}^1))\sharp\letterkapeps{\pi}{s}{\mu},\] then
\[\letterkapeps{\hat a}{s}{\mu}=-\letterkapeps{a}{s}{\mu},\,\letterkapeps{\hat{\gamma}}{s}{\mu}= (\operatorname{p}^1,-\operatorname{p}^2)\sharp\letterkapeps{\gamma}{s}{\mu},\] and, thus,
\begin{equation}\label{lb:incl:super_diff_psi}
	(\letterkapeps{\hat{a}}{s}{\mu},\letterkapeps{\hat{\gamma}}{s}{\mu})\in \partial_{P,\varepsilon}^+\psi(s,\mu).
\end{equation}

Furthermore, notice that $\letterkapeps{\hat{\gamma}}{s}{\mu}$ lies in the negative cone generated by optimal displacements \[\begin{split}
	\operatorname{dis}^+(\mu)\triangleq \Big\{c(\operatorname{Id}-F)^\top:\ \  F:\rd\rightarrow\rd\text{ is an optimal transportation}&{}\\\text{map between }\mu\text{ and }F\sharp\mu,\ \ c\in [0,&+\infty)\Big\}.\end{split}\]

\begin{theorem}\label{lb:th:subsol}
	Let $D_0$ be a bounded domain, and let $D$, $D^{(1)}$ be defined by~\eqref{feedback:intro:D},\eqref{feedback:intro:D_1} respectively. Assume that $\psi\in \operatorname{UC}([0,T]\times D^{(1)})$ satisfies the following property 
	\begin{itemize}
		\item $\psi(T,\mu)\leq G(\mu)$;
		\item there exist constants $C>0$ and $\varepsilon_0>0$ such that, for every $(s,\mu)\in [0,T]\times D^{(1)}$, $\varepsilon\in (0,\varepsilon_0)$, $(a,\gamma)\in \partial_{P,\varepsilon}^+\psi(s,\mu)$ with $\mathscr{b}[\gamma]\in\operatorname{dis}^+(\mu)$, one has that 
		\[a+H(s,\mu,\mathscr{b}[\gamma])\geq -C\varepsilon.\]
	\end{itemize} Then, there exists function for every $(s_*,\mu_*)\in [0,T]\times D_0$,
	\[\operatorname{Val}(s_*,\mu_*)\geq \psi(s_*,\mu_*).\]
\end{theorem} This statement, and inclusion~\eqref{lb:incl:subdiff} immediately imply the following. 
\begin{corollary}\label{lb:corollary:psi_Val} If $\psi\in\operatorname{UC}([0,T]\times\prd)$ is a strict viscosity solution, then, for every $(s_*,\mu_*)\in [0,T]\times\prd$,
	\[\operatorname{Val}(s_*,\mu_*)\geq\psi(s_*,\mu_*).\]
\end{corollary} This, the fact that the value function is uniquely defined and Corollary~\ref{feedback:corollary:bound} imply the following.
\begin{corollary}\label{lb:corollary:uniqueness}
	There exists at most one strict viscosity solution of~\eqref{Bellman:eq:HJ} in the class of functions uniformly continuous on each bounded set.
\end{corollary}

\begin{proof}[Proof of Theorem~\ref{lb:th:subsol}] First, let $\xi\in\mathcal{U}_{s_*,T}$ and $m_\cdot\triangleq m_\cdot[s_*,\mu_*,\xi]$.
	
For each natural $n$, we put 
\begin{equation}\label{lb:intro:s_i}\alpha\triangleq \frac{1}{n}(T-s_*),\ \ s_i\triangleq s_*+i\alpha,\ \ i=0,\ldots,n.\end{equation} Furthermore, let
 $\varkappa>0$ and $\varepsilon\in (0,\varepsilon_0)$ be such that 
 \begin{equation}\label{lb:intro_ineq:vareps}\varepsilon<\alpha^2.\end{equation}
As in the proof of Theorem~\ref{feedback:th:u_kapeps}, we use the designations:
		\begin{itemize}
	\item $\mu_i\triangleq m_{s_i}$;
	\item $t_i\triangleq \letterkapeps{t}{s_i}{\mu_i}$;
	\item $\nu_i\triangleq \letterkapeps{\nu}{s_i}{\mu_i}$;
	\item $\pi_i\triangleq \letterkapeps{\pi}{s_i}{\mu_i}$;
	\item $\hat{a}_i\triangleq \letterkapeps{\hat{a}}{s_i}{\mu_i}=(\letterkapeps{t}{s_i}{\mu_i}-s_i)$;
	\item $\gamma_i\triangleq \letterkapeps{\hat{\gamma}}{s_i}{\mu_i}$.
\end{itemize} Furthermore, let $k$ be the maximal number such that $s_{k-1}\leq \varrho_{1}(\varkappa)$ and $l$ be a minimal number satisfying $T-s_l\leq \varrho_{1}(\varkappa)$. Similarly to formulae~\eqref{feedback:ineq:varphi_kappa_s_l_T},~\eqref{feedback:ineq:L_s_l_T},~\eqref{feedback:ineq:varphi_kappa_s_k_0},~\eqref{feedback:ineq:L_s_0_s_k},
\begin{equation}\label{lb:ineq:psi_kappa_s_l_T}
	|\psi^\varkappa(s_l,\mu_l)-\psi(T,m_T)|\leq \varrho_{4}'(\varkappa),
\end{equation} 
\begin{equation}\label{lb:ineq:L_s_l_T}
	\Bigg|\int_{s_l}^T\int_U L(t,m_t,\xi)\xi(du|t)dt\Bigg|\leq C_3\varrho_{1}(\varkappa).
\end{equation} 
 \begin{equation}\label{lb:ineq:psi_kappa_s_k_0}
 	|\psi(s_0,\mu_0)- \psi(s_k,\mu_k)|\leq \varrho_{5}'(\varkappa,\alpha^*),
 \end{equation} 
	\begin{equation}\label{lb:ineq:L_s_0_s_k}
 	\Bigg|\int_{s_0}^{s_k}\int_U L(t,m_t,\xi)\xi(du|t)dt\Bigg|\leq C_3(\varrho_{1}(\varkappa)+\alpha^*).
 \end{equation} 
 
Following the proof of Theorem~\ref{feedback:th:u_kapeps}, we deduce
 \begin{equation}\label{lb:ineq:varphi_i_third}
 	\begin{split}
 		\psi^\varkappa(s_{i+1}&,\mu_{i+1})+\int_{s_i}^{s_{i+1}}\int_U L(t,m_t,u)\xi(du|t)dt\geq \psi^\varkappa(s_i,\mu_i)\\&+ \hat{a}_i(s_{i+1}-s_i)+\int_{\rd}\mathscr{b}[\hat\gamma_i](y)\cdot \int_{s_i}^{s_{i+1}}\int_U f(t_i,x,\nu_i,u)\xi(du|t)dt\mu_i(dy)\\&+\int_{s_i}^{s_{i+1}}\int_UL(t_i,\mu_i,u)\xi(du|t)dt\\&-(s_{i+1}-s_i)\varrho_{8}'(\varkappa,s_{i+1}-s_i)-(s_{i+1}-s_i)\varrho_{9}'(\varkappa)-\varepsilon\varrho_{2}(\varkappa).
 	\end{split}
 \end{equation} By the Hamiltonian’s definition (see~\eqref{Bellman:intro:Hamiltonian}),
 \[
 \begin{split}
 	\hat{a}_i(s_{i+1}-s_i)+\int_{\rd}\mathscr{b}[\hat\gamma_i](y)\cdot \int_{s_i}^{s_{i+1}}\int_U f(&t_i,x,\nu_i,u)\xi(du|t)dt\mu_i(dy)\\&+\int_{s_i}^{s_{i+1}}\int_UL(t_i,\mu_i,u)\xi(du|t)dt\\ \geq 
 	(s_{i+1}-s_i)\big[\hat{a}_i&+H(t_i,\nu_i,\mathscr{b}[\hat\gamma_i])\big]
 \end{split}
 \] Due to the theorem's assumption \[\hat{a}_i+H(t_i,\nu_i,\mathscr{b}[\hat\gamma_i])\geq -C\varepsilon.\] From~\eqref{lb:ineq:varphi_i_third}, we obtain 
 \begin{equation*}
 	\begin{split}
 		\psi^\varkappa(s_{i+1},\mu_{i+1})+\int_{s_i}^{s_{i+1}}\int_U &L(t,m_t,u)\xi(du|t)dt\geq \psi^\varkappa(s_i,\mu_i)\\&-C\varepsilon (s_{i+1}-s_i)-(s_{i+1}-s_i)\varrho_{8}'(\varkappa,s_{i+1}-s_i)\\&-(s_{i+1}-s_i)\varrho_{9}'(\varkappa)-\varepsilon\varrho_{2}(\varkappa).
 	\end{split}
 \end{equation*} Summing these inequalities and taking into account estimates \eqref{lb:ineq:psi_kappa_s_l_T}--\eqref{lb:ineq:L_s_0_s_k}, we conclude that 
 	\[ \begin{split}
 	J[s_*,\mu_*,\xi]&=\psi(T,m_T)+\int_{s_0}^{s_n}\int_U L(t,m_t,u)\xi(du|t)dt\\&\geq \psi(s_*,\mu_*)-TC\varepsilon-\frac{T}{\alpha}\varepsilon\varrho_{2}(\varkappa)-\varrho_{10}'(\varkappa)-\varrho_{11}'(\varkappa,\alpha)\\ &\geq 
 	 \psi(s_*,\mu_*)-\frac{T^3C}{n^2}-\frac{T^2}{n}\varrho_{2}(\varkappa)-\varrho_{10}'(\varkappa)-\varrho_{11}'(\varkappa,(T-s_*)/n).
 \end{split}\] In the last passage, we choice of partition~\eqref{lb:intro:s_i} and condition~\eqref{lb:intro_ineq:vareps}. Successively passing to the limit as $n\rightarrow\infty$ and $\varkappa\rightarrow 0$, we conclude that 
 \[J[s_*,\mu_*,\xi]\geq \psi(s_*,\mu_*).\] Since $\xi$ is chosen arbitrarily, we obtain the conclusion of the theorem.
\end{proof}


\appendix
\section{Trajectories generated by the continuity equation}\label{sect:app:motion}
\begin{proposition}\label{app:prop:bounds}
	Let $s\in [0,T]$, $\mu\in \prd$, $\xi\in\mathcal{U}_{s,T}$, $m_\cdot=m_\cdot[s,\mu,\xi]$. Then, there exist constants $c_1$, $c_2$, $c_3$ and $c_4$ such that, for every $t\in [s,T]$,
	\begin{enumerate}
		\item\label{control:prop:statement_m_t} $\varsigma(m_t)\leq c_1(1+\varsigma(\mu))$;
		\item $W_2(m_t,\mu)\leq c_2(1+\varsigma(\mu))|t-s|$;
		\item\label{control:prop:statement_X_t} $|X^{s,t}_{m_\cdot}(y)|\leq c_3(1+|y|+\varsigma(\mu))$;
		\item $|X^{s,t}_{m_\cdot}(y)-y|\leq c_4(1+|y|+\varsigma(\mu))|t-s|$
	\end{enumerate}
\end{proposition}
This statement is proved in \cite{Averboukh23}.
\begin{corollary}\label{app:corollary:compact} The set $\{m_t[s,\mu,\xi]:\, \xi\in\mathcal{U}_{s,T},\, t\in [s,T]\}$ is compact in $\prd$.
\end{corollary}
\begin{proof}
	By \cite[Proposition 7.1.5]{ambrosio}, it suffices to show that  
	\[\int_{|x|>R} |x|^2m_t(dx)\rightarrow 0\] as $R\rightarrow 0$ uniformly w.r.t.\ $\xi\in \mathcal{U}_{s,T}$ and $t\in [s,T]$. To this end, notice that 
	\[\int_{|x|>R} |x|^2m_t(dx)=\int_{y:|X^{s,t}_{m_\cdot}(y)|>R} |X^{s,t}_{m_\cdot[s,\mu,\xi]}(y)|^2\mu(dy).\] Furthermore, if $|y|\leq R_0$, then statement~\ref{control:prop:statement_X_t} of Proposition~\ref{app:prop:bounds} gives that 
	$|X^{s,t}_{m_\cdot[s,\mu,\xi]}(y)|\leq c_3(1+R_0+\varsigma(\mu))$. Hence, if, given $R$, we put $R_0(R)\triangleq \frac{1}{c_3}R-1-\varsigma(\mu)$, 
	\[\int_{|x|>R} |x|^2m_t(dx)\leq (c_3)^2\int_{y:\, |y|\geq R_0(R)}(1+|y|+\varsigma(\mu))^2\mu(du).\] The integral in the right-hand side of this inequality tends to zero as $R\rightarrow\infty$ because $\mu\in\prd$. 
\end{proof}

Now we choose $\zeta\in \mathcal{P}(U)$. It is regarded as a constant relaxed control, i.e., we consider on $[s,\theta]$ a control $\xi$ defined by its disintegration as $\xi(\cdot|t)\triangleq \zeta(\cdot)$ for each $t\in [s,\theta]$. As above, we denote $m_\cdot=m_\cdot[s,\mu,\xi]$ and define
\[v^h(y)\triangleq \frac{1}{h}\int_s^{s+h}\int_U f(t,X^{s,t}_{m_\cdot}(y),m_\cdot,u)\zeta(du)dt. \] Notice that by construction, we have that 
$m_{s+h}=(\operatorname{Id}+h v^h)\sharp\mu.$ Finally, we put
\[v(y)\triangleq \int_U f(s,y,\mu,u)\zeta(du)\]

\begin{proposition}\label{app:prop:v_zeta} The following convergence holds true: 
	\[\|v^h-v\|_{\LTwo{\mu}}\rightarrow 0\text{ as }h\rightarrow 0.\]
\end{proposition}
\begin{proof}
	Indeed, due to conditions~\ref{control:cond:f_omega} and~\ref{control:cond:f_Lip}, we have that 
	\[\begin{split}
		|v^h(&y)-v(y)|\\&\leq \frac{1}{h}\int_{s}^{s+h}\big[\omega_f(|s-r|)(1+|y|+\varsigma(\mu)) +C_f(|X^{s,t}_{m_\cdot}-y|+W_2(m_t,\mu))\big]dt.\end{split}\] Here, as above, $C_f$ stands for the Lipschitz constant of the function $f$ w.r.t.\ $x$ and $m$. From Proposition~\ref{app:prop:bounds} and the Minkowski integral inequality, we deduce the desired convergence. 
\end{proof}

\begin{proposition}\label{app:prop:L_limit_zeta} One has that
	\[\frac{1}{h}\int_s^{s+h}\int_U L(t,m_t,u)\zeta(du)dt\rightarrow \int_U L(s,\mu,u)\zeta(du)\text{ as }h\rightarrow 0.\]
\end{proposition} 
\begin{proof}
	The proposition directly follows from the assumption that $L$ is continuous (see assumption~\ref{control:cond:L}) and the fact that $m_t$ for ${t\in [s,T]}$ lies in a compact set. The latter is due to~Corollary~\ref{app:corollary:compact}.
\end{proof} 

Furthermore, for each $h$ we consider a relaxed control $\xi^h\in\mathcal{U}_{s,s+h}$. One can define an averaging of $\xi^h$ w.r.t.\ the time interval $[s,s+h]$ that is a probability $\zeta^h\in \mathcal{P}(U)$ as follows: for each $\psi\in C(U)$,
\[\int_U\psi(u)\zeta^h(du)\triangleq \frac{1}{h}\int_{[s,s+h]\times U}\psi(u)\xi^h(d(t,u)).\] Since $\mathcal{P}(U)$ is compact, there exists a sequence $\{h_n\}_{n=1}^\infty$ and a probability $\zeta\in\mathcal{P}(U)$ such that $h_n\rightarrow 0$, $\zeta^{h_n}\rightarrow \zeta$ as $n\rightarrow\infty$. For simplicity, we denote 
\[\xi^{(n)}\triangleq \xi^{h_n},\, \zeta^{(n)}\triangleq \zeta^{h_n},\, m^{(n)}_\cdot\triangleq m_\cdot[s,\mu,\xi^{(n)}]. \]
Additionally, we put
\[v^{(n)}(y)\triangleq \frac{1}{h^n}\int_s^{s+h_n}\int_U f(t,X^{s,t}_{m^{(n)}_\cdot}(y),m^{(n)}_t,u)\xi^{(n)}(du|t)dt,\] \[ v(y)\triangleq \int_U f(s,y,\mu,u)\zeta(du).\]

\begin{proposition}\label{app:prop:v_xi_zeta_n} One has that 
	\[\|v^{(n)}-v\|_{\LTwo{\mu}}\rightarrow 0\text{ as }n\rightarrow\infty.\]
\end{proposition}
\begin{proof}
	First, for each $n$, we define a velocity field $\tilde{v}^{(n)}$ by the rule:
	\[\tilde{v}^{(n)}\triangleq \int_U f(s,y,\mu,u)\zeta^n(du).\] Conditions~\ref{control:cond:f_omega} and~\ref{control:cond:f_Lip} give that 
	\begin{equation}\label{app:ineq:v_n_tilde_v}
		\|\tilde{v}^{(n)}-v^{(n)}\|_{\LTwo{\mu}}\leq c_5\omega_f(h^n)+c_6h^n
	\end{equation} Here $c_5$ and $c_6$ are constants determined by the probability $\mu$. It remains to prove that $\|\tilde{v}^{(n)}-v\|_{\LTwo{\mu}}$ converges to zero. To this end, we choose $\varepsilon>0$. Since $\mu\in\prd$, there exists $R$ satisfying the following condition 
	\begin{equation}\label{app:ineq:choice_R}\int_{|y|\geq R}(C_1)^2(1+|y|+\varsigma(\mu))^2\mu(dy)\leq \frac{1}{6}\varepsilon^2.\end{equation} Here $C_1$ is a constant from the sublinear growth condition on the function $f$ (see~\eqref{control:ineq:sublinear}). Now let \begin{equation}\label{app:intro:alpha}\alpha\triangleq \varepsilon/\sqrt{27 C_f}.\end{equation} There exists a finite $\alpha$-net for the ball $\mathbb{B}_R(0)\subset \rd$. We denote this net by $\{z_k^\varepsilon\}_{k=1}^K$. By the Lipschitz continuity of the function $f$, one has that, if $y\in \mathbb{B}_\varepsilon(z_k^\varepsilon)$, then 
	\[|f(s,y,\mu,u)-f(s,z_k^\varepsilon,\mu,u)|\leq C_f|y-z_k^\varepsilon|\leq C_f\alpha.\] 
	Now, for each $k$ there exists $N_k$ such that 
	\[\Bigg|\int_U f(s,z_k^\varepsilon,\mu,u)\zeta^{(n)}(du)-\int_U f(s,z_k^\varepsilon,\mu,u)\zeta(du)\Bigg|\leq C_f\alpha.\] Therefore, for $n\geq N\triangleq \max\{N_1,\ldots,N_K\}$ and every $y\in \mathbb{B}_\varepsilon(z_k^\varepsilon)$,
	\[\begin{split}
		\Bigg|\int_U f(s,y,\mu,u)\zeta^{(n)}&(du)-\int_U f(s,y,\mu,u)\zeta(du)\Bigg|\\ \leq 
		\Bigg|\int_U f(s,y,\mu,&u)\zeta^{(n)}(du)-\int_U f(s,z_k^\varepsilon,\mu,u)\zeta^{(n)}(du)\Bigg|\\+\Bigg|\int_U f(s&,y,\mu,u)\zeta(du)-\int_U f(s,z_k^\varepsilon,\mu,u)\zeta(du)\Bigg|\\+
		\Bigg|&\int_U f(s,z_k^\varepsilon,\mu,u)\zeta^{(n)}(du)-\int_U f(s,z_k^\varepsilon,\mu,u)\zeta(du)\Bigg|\leq 3C_f\alpha.
	\end{split}\] Therefore,
	\begin{equation}\label{app:ineq:R_inside}\begin{split}
			\int_{\mathbb{B}_R} |\tilde{v}^{(n)}(y)&-v(y)|^2\mu(dy)\\=
			\int_{\mathbb{B}_R} &\Bigg|\int_U f(s,y,\mu,u)\zeta^{(n)}(du)-\int_U f(s,y,\mu,u)\zeta(du)\Bigg|^2\mu(du)\leq 9C_f\alpha^2
	\end{split}\end{equation} Furthermore, 
	\[
	\begin{split}
		\int_{\rd} |\tilde{v}^{(n)}&(y)-v(y)|^2\mu(dy)\\\leq &\int_{\mathbb{B}_R} |\tilde{v}^{(n)}(y)-v(y)|^2\mu(dy)+2\int_{|y|>R}(|\tilde{v}^{(n)}(y)|^2+|v(y)|^2)\mu(dy)
	\end{split}
	\] Taking into account estimate~\eqref{control:ineq:sublinear} as well as the  choices of $R$ and $\alpha$ (see~\eqref{app:ineq:choice_R},~\eqref{app:intro:alpha}), we conclude that the second term is bounded by $\frac{2}{3}\varepsilon^2$, while the first term due to~\eqref{app:ineq:R_inside} is bounded by $\frac{1}{3}\varepsilon^2$. This means that, given $\varepsilon>0$ one can find $N$ such that, for every $n>N$,
	\[\|\tilde{v}^{(n)}-v\|_{\LTwo{\mu}}\leq \varepsilon.\] This together with~\eqref{app:ineq:v_n_tilde_v} gives the desired convergence.
\end{proof}
\begin{proposition}\label{app:prop:L_xi_zeta} The following convergence holds true:
	\[\frac{1}{h_n}\int_s^{s+h_n}\int_U L(t,m^{(n)}_t,u)\xi^{(n)}(du|t)dt\rightarrow \int_U L(t,m^{(n)}_t,u)\zeta(du).\] 
\end{proposition}
\begin{proof}
	Recall that $L$ is continuous, and $W_2(m^{(n)}_t,\mu)\leq c_7h$, where $c_7$ is a constant dependent on $\mu$ (see Proposition~\ref{app:prop:bounds}). Hence, for each $t\in [s,s+h_n]$,
	\[|L(t,m^{(n)}_t,u)-L(s,\mu,u)|\leq \beta_n,\] where $\beta_n\rightarrow 0$ as $n\rightarrow\infty$. This implies
	\[\begin{split}
		\Bigg|\frac{1}{h_n}&\int_s^{s+h_n}\int_U L(t,m^{(n)}_t,u)\xi^{(n)}(du|t)dt- \int_U L(t,m^{(n)}_t,u)\zeta(du)\Bigg|\\&\leq \Bigg|\frac{1}{h_n}\int_s^{s+h_n}\int_U L(s,\mu,u)\xi^{(n)}(du|t)dt- \int_U L(t,m^{(n)}_t,u)\zeta(du)\Bigg|+\beta_n.\end{split}\] The first term tends to zero due to the very definition of the probabilities $\zeta^{(n)}$ and the fact that the sequence $\{\zeta^{(n)}\}_{n=1}^\infty$ converges narrowly to $\zeta$.
\end{proof}

\bibliography{feedback_cont_3.bib}

\end{document}